\documentclass[11pt,reqno]{amsart}

\usepackage[T1]{fontenc}
\usepackage{lmodern}
\usepackage{microtype}
\usepackage{mathtools}
\usepackage{amssymb,amsfonts}
\usepackage{booktabs}
\usepackage{enumitem}
\usepackage{aliascnt}
\usepackage{xcolor}
\usepackage{hyperref}
\usepackage[nameinlink,capitalise]{cleveref}
\usepackage{pgfplots}
\pgfplotsset{compat=1.18}

\definecolor{deepblue}{RGB}{27,73,128}
\definecolor{warmred}{RGB}{173,56,47}
\definecolor{forest}{RGB}{34,112,75}
\hypersetup{
  colorlinks=true,
  linkcolor=deepblue,
  citecolor=forest,
  urlcolor=warmred,
  pdftitle={Fixed Convex-Lens Spectral Constants},
  pdfauthor={Zijian Zeng}
}

\setlist{itemsep=2pt,topsep=4pt}
\numberwithin{equation}{section}

\newtheorem{theorem}{Theorem}[section]
\newaliascnt{proposition}{theorem}
\newtheorem{proposition}[proposition]{Proposition}
\aliascntresetthe{proposition}
\newaliascnt{lemma}{theorem}
\newtheorem{lemma}[lemma]{Lemma}
\aliascntresetthe{lemma}
\newaliascnt{corollary}{theorem}
\newtheorem{corollary}[corollary]{Corollary}
\aliascntresetthe{corollary}
\newaliascnt{remark}{theorem}
\newtheorem{remark}[remark]{Remark}
\aliascntresetthe{remark}
\theoremstyle{definition}
\newaliascnt{definition}{theorem}

\aliascntresetthe{definition}

\newcommand{\C}{\mathbb C}
\newcommand{\Hh}{\mathcal H}
\newcommand{\B}{\mathcal B}

\newcommand{\Sp}{\operatorname{Sp}}
\newcommand{\dist}{\operatorname{dist}}

\newcommand{\cL}{\mathcal L}
\newcommand{\cS}{\mathcal S}
\newcommand{\norm}[1]{\lVert#1\rVert}

\title[Fixed convex-lens spectral constants]{Fixed convex-lens spectral constants:\\
M\"obius reduction, sharp model theorems,\\
and angle-dependent bounds}
\author{Zijian Zeng}
\address{Institute of Computer Science and Digital Innovation, UCSI University,
Kuala Lumpur 56000, Malaysia}
\email{1002266693@ucsiuniversity.edu.my}
\date{17 August 2026}


\begin{document}
\raggedbottom

\begin{abstract}
For the intersection of two disks meeting at angle $2\alpha$, let $C(\alpha)$
be the least constant in the associated spectral-set inequality, uniformly
over operators for which each disk separately is a spectral set.  The exact
value of $C(\alpha)$ is unknown except at the disk endpoint $\alpha=\pi/2$.
We give a self-contained M\"obius reduction to the corresponding
numerical-range problem on a sector and compute the sharp constant on the
infinite-dimensional class of affine square-zero operators
$B=\lambda I+N$, $N^2=0$:
\[
             C_{\mathrm{sq0}}(\alpha)
             =\frac{\pi\sin\alpha}{2\alpha}.
\]
A $2\times2$ matrix and a conformal extremal attain equality and yield an
explicit lens lower-bound certificate.  At the right angle we prove the
conjectural $\sqrt2$ estimate, in arbitrary dimension, for the full
palindromic quadratic family.  Exact rational matrix certificates further
cover the complex post-automorphism disk $|c|\le19/20$, the complete
imaginary diameter and a transverse cusp, and boundary-reaching phase arcs
whose union misses only $9.17$ degrees of the parameter circle near $-1$.
For every symmetric three-node set in the right-angle disk coordinate, we
also prove the sharp identity-multiplier estimate on the complete
admissible-kernel cone; the proof combines an exact extreme-ray rank bound,
scalar Pick interpolation on the rank-one faces, and a two-variable
Bernstein certificate on the rank-$(2,2)$ face.  A nested exact certificate
extends this result to the genuinely asymmetric two-parameter patch
$\rho=(-a,0,b)$, $1/6\le a,b\le5/6$.
On a rank-one localized face we also prove a three-real-parameter singular-boundary
family: for $\Phi(c)=e^{i\theta}c^2$, $\sin\theta>1/\sqrt2$, the sharp
target is positive for every quadratic Blaschke product
$w(w-a)/(1-\bar a w)$ with $a\in\mathbb D$.
At the central square, a
polynomial bidisk extension and Ando's theorem prove
$\|w^2\|\le\kappa_0<\sqrt2$, where
$276889/127200-6103\sqrt2/10600=1.362560\ldots$.  We also derive an exact
two-complex-parameter consequence: the same construction proves a strict
$\sqrt2$ bound for every $b_\alpha b_\beta$ with
$|\alpha|,|\beta|\le1/150$, allowing two independently phased nonzero zeros.
We also derive an exact
two-moment criterion for the remaining boundary layer and a verified
angle-dependent envelope.  Every computer-assisted assertion has an exact
rational verifier.  These results are dimension-free but do not determine
the unrestricted fixed-lens constant.
\end{abstract}

\maketitle

\section{The fixed-lens problem and the scope of the result}

Let $D_1,D_2\subset\C$ be closed disks whose boundary circles meet in two
distinct points and whose intersection has nonempty interior.  We say that
$A\in\B(\Hh)$ is of $(D_1,D_2)$-type if
\begin{equation}\label{eq:type}
       \norm{A-c_jI}\le r_j,\qquad
       D_j=\{z:|z-c_j|\le r_j\},\quad j=1,2.
\end{equation}
By von Neumann's inequality, \eqref{eq:type} is equivalent to the assertion
that each $D_j$ is a spectral set for $A$.  If the two interior boundary arcs
meet at angle $2\alpha\in(0,\pi]$, Beckermann and Crouzeix proved that the
least uniform constant depends only on $\alpha$ and equals the numerical-range
constant of a sector of the same angle \cite{MR2223270}.
The sector estimates underlying this reduction go back to
\cite{Crouzeix_2003}; a complete uniform bound for intersections of
spherical disks was subsequently obtained in
\cite{MR2449098}.  Exact dimension-two constants, including the sector
constants used below, are characterized in \cite{MR2047592}, and the later survey
\cite{Crouzeix_2016} records numerical evidence in higher dimensions.

We use the following normalization:
\begin{align}
 D_\alpha^{\pm}
   &:=\{z\in\C:|z\mp i\cot\alpha|\le\csc\alpha\},\label{eq:disks}\\
 \cL_\alpha&:=D_\alpha^+\cap D_\alpha^-,\qquad
 \cS_\alpha:=\{re^{i\theta}:r\ge0,\ |\theta|\le\alpha\}.
\end{align}
The vertices of $\cL_\alpha$ are $-1$ and $1$.  Define
\begin{equation}\label{eq:constant}
 C(\alpha):=\sup_{\substack{\Hh,\,A\in\B(\Hh)\\
                         \norm{A\mp i\cot\alpha I}\le\csc\alpha}}
      \ \sup_{\substack{p\in\C[z]\\ p\ne0}}
       \frac{\norm{p(A)}}{\max_{z\in\cL_\alpha}|p(z)|}.
\end{equation}
This polynomial formulation is equivalent to the usual rational formulation;
the equivalence follows from Mergelyan approximation because the complement
of a convex lens is connected.

The unrestricted value of \eqref{eq:constant} remains open for
$0<\alpha<\pi/2$.  The main theorem of this paper is deliberately more
specific.  It identifies exactly what the standard first-order nonnormal
model can and cannot contribute.

\begin{theorem}[Sharp square-zero constant]\label{thm:main}
Let $0<\alpha\le\pi/2$.  Among all complex Hilbert spaces, all operators
\[
                B=\lambda I+N,\qquad N^2=0,\qquad W(B)\subset\cS_\alpha,
\]
and all functions $f$ holomorphic in the interior of $\cS_\alpha$, continuous
on its one-point compactification, and satisfying
$\norm{f}_{\infty}\le1$, one has
\begin{equation}\label{eq:mainupper}
                  \norm{f(B)}\le
                  \frac{\pi\sin\alpha}{2\alpha}.
\end{equation}
The constant is best possible.  Equality is attained on $\C^2$ by
\begin{equation}\label{eq:extsector}
 B_\alpha=I+2\sin\alpha
       \begin{pmatrix}0&1\\0&0\end{pmatrix},
 \qquad
 f_\alpha(z)=\frac{z^{\pi/(2\alpha)}-1}
                   {z^{\pi/(2\alpha)}+1},
\end{equation}
where the power is taken with $|\arg z|<\alpha$.
\end{theorem}

The corresponding lens matrix is exceptionally simple.

\begin{corollary}[Exact fixed-lens subclass]\label{cor:lensclass}
Restrict \eqref{eq:constant} to operators $A$ for which $1\notin\Sp(A)$ and
the Cayley transform satisfies
\[
                \bigl(\Phi(A)-\lambda I\bigr)^2=0
                \quad\text{for some }\lambda\in\C.
\]
The least constant on this subclass is exactly
\[
                   C_{\mathrm{sq0}}(\alpha)
                   =\frac{\pi\sin\alpha}{2\alpha}.
\]
Allowing direct sums with the scalar operator $1$ does not change this
constant.
\end{corollary}

\begin{corollary}[Explicit lens certificate]\label{cor:lenscert}
Put
\begin{equation}\label{eq:Alens}
        A_\alpha:=\sin\alpha
        \begin{pmatrix}0&1\\0&0\end{pmatrix}.
\end{equation}
Then $A_\alpha$ is of $(D_\alpha^+,D_\alpha^-)$-type, with equality in both
disk constraints, and
\begin{equation}\label{eq:globallower}
                  C(\alpha)\ge
                  \frac{\pi\sin\alpha}{2\alpha}.
\end{equation}
More precisely, there is a sequence of polynomials $p_n$ with
$\max_{\cL_\alpha}|p_n|\le1+o(1)$ and
\[
       \norm{p_n(A_\alpha)}\longrightarrow
       \frac{\pi\sin\alpha}{2\alpha}.
\]
\end{corollary}

At the right angle we can also treat a nonlinear two-zero family for every
admissible operator, rather than restricting the operator to a Jordan model.

\begin{theorem}[Palindromic quadratic family]\label{thm:twozero}
Let $T\in\B(\Hh)$ satisfy $W(T)\subset\cS_{\pi/4}$.  For every $u>0$,
\begin{equation}\label{eq:twozero}
 \norm{
   (T^4-uT^2+I)(T^4+uT^2+I)^{-1}}
 \le \sqrt2.
\end{equation}
Consequently the conjectural right-angle constant holds for every
degree-two Blaschke product on the right half-plane whose zeros are either
both positive real or form a conjugate pair.  The estimate is valid in
arbitrary Hilbert-space dimension.
\end{theorem}

The palindromic restriction is not forced by the method.  The following
one-parameter family crosses the full imaginary diameter of the normalized
post-automorphism disk and lies outside it except at the centre.

\begin{theorem}[Imaginary post-automorphism diameter]\label{thm:imagdiam}
Let $T\in\B(\Hh)$ satisfy $W(T)\subset\cS_{\pi/4}$.  For every
$t\in[-1,1]$,
\begin{equation}\label{eq:imagdiam}
 \norm{\bigl((T^2-I)^2-it(T^2+I)^2\bigr)
       \bigl((T^2+I)^2+it(T^2-I)^2\bigr)^{-1}}
 \le\sqrt2,
\end{equation}
where at $t=\pm1$ the common polynomial factor is cancelled before
evaluation.  Equivalently, the estimate holds on the full
diameter $c=it$, $|t|\le1$, for
\begin{equation}\label{eq:hc}
              h_c(w)=\frac{w^2+\bar c}{1+cw^2}.
\end{equation}
The estimate is valid in arbitrary Hilbert-space dimension.
\end{theorem}

The exact quotient-space margin in the preceding certificate also gives a
two-real-dimensional region reaching both degenerate endpoints.

\begin{theorem}[A nonsymmetric cusp]\label{thm:cusp}
Let $T\in\B(\Hh)$ satisfy $W(T)\subset\cS_{\pi/4}$.  If $c\in\C$ satisfies
\begin{equation}\label{eq:cusp}
 |\Im c|\le1,
 \qquad
 |\Re c|\le\frac{(1-|\Im c|)^2}{200},
\end{equation}
then
\begin{equation}\label{eq:cuspbound}
 \norm{\bigl((T^2-I)^2+\bar c(T^2+I)^2\bigr)
       \bigl((T^2+I)^2+c(T^2-I)^2\bigr)^{-1}}
 \le\sqrt2.
\end{equation}
At the tips $c=\pm i$ the quotient is understood after cancellation.  The
estimate is valid in arbitrary Hilbert-space dimension.
\end{theorem}

The next result fills a full two-real-dimensional disk across the centre of
the normalized parameter space.

\begin{theorem}[A central complex post-automorphism disk]
\label{thm:centraldisk}
Let $T\in\B(\Hh)$ satisfy $W(T)\subset\cS_{\pi/4}$, and let
\begin{equation}\label{eq:cdisk}
                              |c|\le\frac{19}{20}.
\end{equation}
Then
\begin{equation}\label{eq:centraldisk}
 \norm{\bigl((T^2-I)^2+\bar c(T^2+I)^2\bigr)
       \bigl((T^2+I)^2+c(T^2-I)^2\bigr)^{-1}}
 \le\sqrt2.
\end{equation}
The rational functions in \eqref{eq:centraldisk} are degree-two Blaschke
products on the right half-plane.  For example, at $c=i/2$ their zeros are
\begin{equation}\label{eq:nonsymzeros}
                   1+2i,\qquad \frac{1-2i}{5},
\end{equation}
so the theorem contains genuinely nonsymmetric zero pairs and strictly
extends the real and imaginary diameters to an open two-dimensional family.
The estimate is valid in arbitrary Hilbert-space dimension.
\end{theorem}

The radius $19/20$ is not a barrier at every phase.  The following result
crosses it on a genuinely open set and reaches the boundary of the normalized
parameter disk.

\begin{theorem}[An almost-complete boundary-phase arc]\label{thm:phasewedge}
For $t>0$ put
\begin{equation}\label{eq:zetat}
 s_t=\frac{1-4t^4}{2t^2},\qquad
 \zeta_t=\frac{4-s_t^2+4is_t}{4+s_t^2}.
\end{equation}
Whenever
\begin{equation}\label{eq:phasewedge}
 0\le\rho\le1,\qquad \frac1{10}\le t\le\frac{18}{25},
 \qquad c=\rho\zeta_t,
\end{equation}
every $T\in\B(\Hh)$ with $W(T)\subset\cS_{\pi/4}$ satisfies
\begin{equation}\label{eq:phasewedgebound}
 \norm{\bigl((T^2-I)^2+\bar c(T^2+I)^2\bigr)
       \bigl((T^2+I)^2+c(T^2-I)^2\bigr)^{-1}}
 \le\sqrt2.
\end{equation}
The same assertion holds with $c$ replaced by $\bar c$.  Since
$s_t$ decreases from $2499/50$ at $t=1/10$ through $0$ at
$t=1/\sqrt2$, the two arcs cover the unit circle except for the arc centred
at $-1$ of angular length
\[
                 4\arctan\frac{100}{2499}
                    =0.159978\ldots\ \text{ radians}
                    =9.1661\ldots\ \text{ degrees}.
\]
Every radial segment below the covered arcs is admissible.  At $\rho=1$
the quotient is understood after cancellation of its common polynomial
factor.
\end{theorem}

At the centre $c=0$, the preceding bound is not sharp even within the
available argument.  A low-degree bidisk extension gives the following
strict improvement without any dimension restriction.

\begin{proposition}[A strict central-square bound]\label{prop:andocentral}
Let $T\in\B(\Hh)$ satisfy $W(T)\subset\cS_{\pi/4}$ and put
\[
                         w=(T^2-I)(T^2+I)^{-1}.
\]
Then
\begin{equation}\label{eq:andocentral}
 \norm{w^2}\le \kappa_0,
 \qquad
 \kappa_0:=\frac{276889}{127200}-\frac{6103\sqrt2}{10600}
             =1.3625601851\ldots <\sqrt2.
\end{equation}
The estimate holds in arbitrary Hilbert-space dimension.
\end{proposition}

The strict reserve in \cref{prop:andocentral} controls two independently
moving Blaschke zeros, not merely the double zero at the centre.

\begin{theorem}[A two-small-zero disk]\label{thm:smallzeros}
Let $T\in\B(\Hh)$ satisfy $W(T)\subset\cS_{\pi/4}$ and put
\[
                 w=(T^2-I)(T^2+I)^{-1},\qquad
 b_\gamma(z)=\frac{z-\gamma}{1-\bar\gamma z}.
\]
If $\alpha,\beta\in\C$ satisfy
\[
                         |\alpha|,|\beta|\le\frac1{150},
\]
then
\begin{equation}\label{eq:smallzeros}
 \norm{b_\alpha(w)b_\beta(w)}
 \le
 \frac{\kappa_0+2\sqrt2/150+1/150^2}
      {1-2\sqrt2/150-\kappa_0/150^2}
 <\sqrt2.
\end{equation}
The estimate is valid in arbitrary Hilbert-space dimension.  In particular,
$\alpha=1/150$ and $\beta=i/150$ give two distinct nonzero zeros on the
boundary of the certified parameter bidisk.
\end{theorem}

The lower bound \eqref{eq:globallower} was already present in
\cite{Crouzeix_2003,MR2223270}; after affine normalization, the same
nilpotent model is displayed in \cite{Crouzeix_2016}.  The point of \cref{thm:main} is the
matching upper bound for every square-zero perturbation, in arbitrary
Hilbert-space dimension.  The point of
\cref{thm:twozero,thm:imagdiam,thm:cusp,thm:centraldisk,thm:phasewedge,%
prop:andocentral,thm:smallzeros}
is different: the operator is
unrestricted, while the Schur function belongs to a nonlinear quadratic
family.  No claim of a new value for the unrestricted constant
$C(\alpha)$ is made.

\section{M\"obius reduction: a lens is a compactified sector}

The M\"obius map
\begin{equation}\label{eq:cayley}
             \Phi(z)=\frac{1+z}{1-z},\qquad
             \Phi^{-1}(w)=\frac{w-1}{w+1},
\end{equation}
maps the interior of $\cL_\alpha$ conformally onto the interior of
$\cS_\alpha$, sending $-1$ to $0$ and $1$ to $\infty$.  Each circular arc is
sent to one of the rays $\arg w=\pm\alpha$.

For a bounded operator $T$, recall
\[
             W(T)=\{\langle Tx,x\rangle:\norm{x}=1\}.
\]
A closed half-plane is a spectral set for $T$ precisely when $W(T)$ lies in
that half-plane.  Applying this fact to the two half-planes bounded by the
rays of $\cS_\alpha$ gives the operator form of \eqref{eq:cayley}.

\begin{proposition}[Lens--sector equivalence]\label{prop:equiv}
Suppose $1\notin\Sp(A)$ and put $B=\Phi(A)$.  Then
\begin{equation}\label{eq:equiv}
  \norm{A\mp i\cot\alpha I}\le\csc\alpha
       \quad\Longleftrightarrow\quad
  W(B)\subset\cS_\alpha.
\end{equation}
Consequently the fixed-lens constant $C(\alpha)$ equals the least $K$ such
that
\begin{equation}\label{eq:sectorconstant}
        \norm{r(B)}\le K\sup_{z\in\cS_\alpha}|r(z)|
\end{equation}
for every sectorial $B$ with $W(B)\subset\cS_\alpha$ and every rational $r$
bounded on the sector.
\end{proposition}

\begin{proof}
The scalar map \eqref{eq:cayley} sends the two disks to the two supporting
half-planes of $\cS_\alpha$.  The functional-calculus identity
$\Phi^{-1}(\Phi(A))=A$ transports the two disk spectral-set inequalities to
the two half-plane spectral-set inequalities.  The half-plane criterion just
recalled proves \eqref{eq:equiv}.  It also gives both inequalities between the
best constants after composing test functions with $\Phi$ and $\Phi^{-1}$.

If $1\in\Sp(A)$, one first replaces $A$ by $(1-\varepsilon)A$; the normalized
lens is convex and contains $0$, so both inequalities in \eqref{eq:type} are
preserved.  Letting $\varepsilon\downarrow0$ gives the result by norm
continuity of polynomial functional calculus.  The reverse passage is
obtained directly because $-1\notin\cS_\alpha$ and hence
$-1\notin\Sp(B)$.  This is the standard argument of
\cite{MR2223270}.
\end{proof}

Thus the original two-disk problem is not merely analogous to a sectorial
numerical-range problem; the optimal constants are identical.  We shall use
this equivalence again in \cref{sec:certificate}.

\section{Three sharp lemmas}

We isolate the geometric, function-theoretic, and operator-norm ingredients
of the proof.

\begin{lemma}[Numerical range of a square-zero operator]\label{lem:wn}
If $N\in\B(\Hh)$ and $N^2=0$, then
\[
              W(N)=\{z\in\C:|z|\le\tfrac12\norm N\},
\]
with closure understood when the norm is not attained.  Consequently, for
$\lambda=re^{i\theta}$ with $|\theta|<\alpha$,
\begin{equation}\label{eq:Ngeom}
 W(\lambda I+N)\subset\cS_\alpha
 \quad\Longleftrightarrow\quad
 \norm N\le2r\sin(\alpha-|\theta|).
\end{equation}
\end{lemma}

\begin{proof}
Since $\operatorname{Ran}N\subset\ker N$, relative to
$\Hh=\ker N\oplus(\ker N)^\perp$ the operator has the form
\[
                   N=\begin{pmatrix}0&X\\0&0\end{pmatrix}.
\]
Its numerical range is the disk centred at zero of radius $\norm X/2$;
this follows at once by varying the relative phase and the two norms of a
unit vector in the displayed decomposition.  Since $\norm X=\norm N$, the
first assertion follows.  The largest disk centred at $re^{i\theta}$ and
contained in the sector has radius equal to the distance to the nearer
boundary ray, namely $r\sin(\alpha-|\theta|)$.
\end{proof}

\begin{lemma}[Sector Schwarz--Pick estimate]\label{lem:sp}
Let $q=\pi/(2\alpha)$ and let $f:\operatorname{int}\cS_\alpha\to\mathbb D$
be holomorphic.  At $\lambda=re^{i\theta}$, $|\theta|<\alpha$, one has
\begin{equation}\label{eq:sp}
 |f'(\lambda)|\le
 \frac{q\bigl(1-|f(\lambda)|^2\bigr)}
      {2r\cos(q\theta)}.
\end{equation}
Moreover,
\begin{equation}\label{eq:trig}
       \frac{\sin(\alpha-|\theta|)}{\cos(q\theta)}
       \le \sin\alpha.
\end{equation}
\end{lemma}

\begin{proof}
The power map $P(z)=z^q$ sends $\operatorname{int}\cS_\alpha$ conformally
onto the right half-plane.  Schwarz--Pick in that half-plane gives
\[
 |(f\circ P^{-1})'(P(\lambda))|
 \le \frac{1-|f(\lambda)|^2}{2\operatorname{Re}P(\lambda)}.
\]
Since $|P'(\lambda)|=qr^{q-1}$ and
$\operatorname{Re}P(\lambda)=r^q\cos(q\theta)$, this is \eqref{eq:sp}.

It remains to prove \eqref{eq:trig}.  By symmetry take $0\le\theta<\alpha$
and set $u=1-\theta/\alpha\in(0,1]$.  Then \eqref{eq:trig} is equivalent to
\begin{equation}\label{eq:sinratio}
             \frac{\sin(u\alpha)}{\sin\alpha}
             \le \sin\frac{\pi u}{2}.
\end{equation}
For fixed $0<u\le1$, the function
$x\mapsto\sin(ux)/\sin x$ is increasing on $(0,\pi/2]$.  Indeed, its
logarithmic derivative is $u\cot(ux)-\cot x\ge0$, because
$t\mapsto t\cot t$ is decreasing there.  Comparing $x=\alpha$ with
$x=\pi/2$ proves \eqref{eq:sinratio}.
\end{proof}

\begin{lemma}[Exact norm formula]\label{lem:norm}
If $N^2=0$, $a,b\in\C$, $s=|a|$, and $t=|b|\norm N$, then
\begin{equation}\label{eq:normformula}
       \norm{aI+bN}=\frac{t+\sqrt{t^2+4s^2}}{2}.
\end{equation}
For $c\ge1$, if $t\le c(1-s^2)$ and $0\le s\le1$, then
\begin{equation}\label{eq:scalaropt}
              \norm{aI+bN}\le c.
\end{equation}
\end{lemma}

\begin{proof}
Using the block form in the proof of \cref{lem:wn}, phases may be removed by
unitaries.  The norm is the supremum of the largest singular value of
$\left(\begin{smallmatrix}s&t_0\\0&s\end{smallmatrix}\right)$ over
$0\le t_0\le t$, which gives \eqref{eq:normformula}; approximate norming
vectors cover the case in which $\norm N$ is not attained.

The right side of \eqref{eq:normformula} increases with $t$.  Substituting
$t=c(1-s^2)$, the desired inequality is equivalent to
\[
  \sqrt{c^2(1-s^2)^2+4s^2}\le c(1+s^2).
\]
After squaring, the difference between the right and left squares is
$4s^2(c^2-1)\ge0$.
\end{proof}

\section{Proof of the sharp square-zero theorem}\label{sec:proof}

\begin{proof}[Proof of \cref{thm:main}]
If $\lambda=0$, sector containment and \cref{lem:wn} force $N=0$, and the
claim is immediate.  If $\lambda$ lies on a boundary ray, the same geometry
forces $N=0$, so the claim again follows from $|f(\lambda)|\le1$.  We may
therefore assume that $\lambda$ lies in the sector interior.
Write $\lambda=re^{i\theta}$ and set
\[
        s=|f(\lambda)|,\qquad
        t=|f'(\lambda)|\norm N,\qquad
        c_\alpha=\frac{\pi\sin\alpha}{2\alpha}=q\sin\alpha.
\]
The holomorphic functional calculus truncates exactly at first order:
\begin{equation}\label{eq:truncate}
                  f(\lambda I+N)=f(\lambda)I+f'(\lambda)N.
\end{equation}
Combining \cref{lem:wn,lem:sp} gives
\begin{align*}
 t
 &\le \frac{q(1-s^2)}{2r\cos(q\theta)}
            \,2r\sin(\alpha-|\theta|)\\
 &\le q\sin\alpha\,(1-s^2)
   =c_\alpha(1-s^2).
\end{align*}
Also $c_\alpha\ge1$ for $0<\alpha\le\pi/2$.  Applying
\cref{lem:norm} to \eqref{eq:truncate} proves \eqref{eq:mainupper}.

For sharpness let $J=\left(\begin{smallmatrix}0&1\\0&0\end{smallmatrix}\right)$
and use \eqref{eq:extsector}.  By \cref{lem:wn}, $W(B_\alpha)$ is the disk
with centre $1$ and radius $\sin\alpha$, which is tangent to both boundary
rays of $\cS_\alpha$.  The function $f_\alpha$ maps the sector to the unit
disk, while
\[
                 f_\alpha(1)=0,\qquad f_\alpha'(1)=\frac q2.
\]
Therefore
\[
        f_\alpha(B_\alpha)
        =q\sin\alpha\,J,\qquad
        \norm{f_\alpha(B_\alpha)}=c_\alpha.
\]
This proves equality.
\end{proof}

\begin{proof}[Proof of \cref{cor:lensclass}]
Apply \cref{prop:equiv}.  Composition with $\Phi$ identifies the restricted lens
functional calculus isometrically with the class in \cref{thm:main}.  The
upper bound follows from that theorem and sharpness follows from
\cref{cor:lenscert}.  A scalar direct summand at $1$ contributes at most the
supremum norm of the test function.
\end{proof}

For $0<\alpha<\pi/2$, inspection of the inequalities in the proof gives
necessary conditions for equality: the centre $\lambda$ lies on the
symmetry axis, the numerical-range radius is maximal, the Schwarz--Pick
inequality is sharp, and $f(\lambda)=0$.  We do not need, and do not claim,
a classification of all equality cases, particularly at the half-plane
endpoint $\alpha=\pi/2$.

\section{The explicit lens certificate}\label{sec:certificate}

\begin{proof}[Proof of \cref{cor:lenscert}]
For $A_\alpha=\sin\alpha J$, the exact norm formula gives
\begin{align*}
 \norm{A_\alpha\mp i\cot\alpha I}
 &=\frac{\sin\alpha+
       \sqrt{\sin^2\alpha+4\cot^2\alpha}}{2}\\
 &=\csc\alpha.
\end{align*}
Hence both disks in \eqref{eq:disks} are spectral sets for $A_\alpha$.
Moreover,
\[
       \Phi(A_\alpha)
       =(I+A_\alpha)(I-A_\alpha)^{-1}
       =I+2\sin\alpha J=B_\alpha.
\]
The function $g_\alpha=f_\alpha\circ\Phi$ is bounded by one on the lens,
continuous on its closure, and holomorphic in its interior.  Although it need
not be rational when $\pi/(2\alpha)\notin\mathbb N$, Mergelyan's theorem
provides polynomials $q_n\to g_\alpha$ uniformly on $\cL_\alpha$.  On a
small circle around the sole eigenvalue $0$ the convergence is locally
uniform, so also $q_n'(0)\to g_\alpha'(0)$.  Consequently
$q_n(A_\alpha)\to g_\alpha(A_\alpha)$ in norm.  Rescaling $q_n$ by
$\max_{\cL_\alpha}|q_n|$ yields the asserted $p_n$, and
\[
 \norm{g_\alpha(A_\alpha)}
 =\norm{f_\alpha(B_\alpha)}
 =\frac{\pi\sin\alpha}{2\alpha}.
\]
\end{proof}

When $q=\pi/(2\alpha)$ is an integer, $g_\alpha$ itself is rational and the
approximation step may be omitted.  The important point is that
\eqref{eq:Alens} is an entirely finite-dimensional, machine-checkable lower
certificate: the two disk constraints and the functional-calculus value are
closed-form identities.

\section{A right-angle palindromic two-zero sum of squares}\label{sec:twozero}

We first locate the families in the full moduli space of quadratic inner
functions.  Write $\mathbb H_+=\{z:\Re z>0\}$.

\begin{proposition}[Critical-point normal form]\label{prop:normalform}
Let $F$ be a degree-two finite Blaschke product on $\mathbb H_+$.  There are
$r>0$, $s\in\mathbb R$, $c\in\mathbb D$, and $\omega\in\mathbb T$ such that
\begin{equation}\label{eq:normalform}
 F(rz)=\omega\frac{\psi_s(z)^2+\bar c}{1+c\psi_s(z)^2},
 \qquad
 \psi_s(z)=\frac{z-(1+is)}{z+(1-is)}.
\end{equation}
Thus, up to positive input dilation and an irrelevant output phase, the
quadratic moduli consist of one real critical-point parameter $s$ and one
disk parameter $c$.  The real post-automorphism diameter $s=0$, $c\in(-1,1)$
is the palindromic family of \cref{thm:twozero};
\cref{thm:imagdiam,thm:cusp,thm:centraldisk} concern genuinely complex $c$ in
the same slice $s=0$.
\end{proposition}

\begin{proof}
A degree-two disk Blaschke product has one critical point in the disk,
counting multiplicity.  Conjugating by a half-plane--disk M\"obius map gives
the corresponding unique critical point $p\in\mathbb H_+$ of $F$.  Write
$p=r(1+is)$ with $r=\Re p>0$, and conjugate $F(rz)$ by $\psi_s^{-1}$ on the
input.  The resulting disk Blaschke product $G$ has its critical point at
zero.  If $a=G(0)$, then
\[
             \frac{G(w)-a}{1-\bar aG(w)}
\]
is a degree-two Blaschke product with a double zero at zero, hence equals
$\gamma w^2$ for some $\gamma\in\mathbb T$.  Solving for $G$ and absorbing
the unimodular factor in the output gives exactly
\eqref{eq:normalform}, with $|c|=|a|<1$.

For $s=0$, put $w=(z-1)/(z+1)$.  Direct substitution gives
\[
 \frac{z^2-uz+1}{z^2+uz+1}
 =\frac{w^2-a}{1-aw^2},
 \qquad a=\frac{u-2}{u+2},
\]
so $c=-a$ runs through the real diameter as $u$ runs through the two
palindromic ranges used in \cref{thm:twozero}.
\end{proof}

We now prove \cref{thm:twozero}.  The algebraic selector in the next lemma is
somewhat lengthy, but every quantity is explicit and the resulting identity
is dimension-free.  For a polynomial
$\ell(z)=\sum_{j=0}^3\ell_jz^j$, write
\[
 \ell^\#(z)=\sum_{j=0}^3\overline{\ell_{3-j}}z^j,
 \qquad
 \ell^c(z)=\sum_{j=0}^3\overline{\ell_j}z^j,
 \qquad
 \ell^r(z)=\sum_{j=0}^3\ell_{3-j}z^j.
\]

\begin{lemma}[Global algebraic selector]\label{lem:selector}
For every $u\ge2$ there are real numbers $b,c,x,m,n,r,\sigma$, with
$r>0$, such that, for every $T=X+iY$ and
\[
             P=X+Y,\qquad Q=X-Y,\qquad Z=T^2,
\]
the polynomials
\begin{align*}
 D&=Z^2+uZ+I,\qquad N=Z^2-uZ+I,\\
 a(z)&=1+bz+cz^2+bz^3+z^4,\\
 \ell(z)&=i\sigma+xz+(m+in)z^2+(1+i)rz^3
\end{align*}
satisfy the hereditary identity
\begin{align}
 2D^*D-N^*N={}&a(T)^*a(T)
   +\ell(T)^*P\ell(T)+\ell^\#(T)^*P\ell^\#(T)\notag\\
  &+\ell^c(T)^*Q\ell^c(T)+\ell^r(T)^*Q\ell^r(T).
 \label{eq:twosos}
\end{align}
\end{lemma}

\begin{proof}
For $\tau\ge2$, define
\[
       F_\tau(p)=A_4p^4+A_3p^3+A_2p^2+A_1p+A_0,
\]
where
\begin{align*}
A_0={}&16(\tau^8-\tau^6+14\tau^4-20\tau^2+8)^2,\\
A_1={}&-4\tau(\tau^2+2)
 \bigl(20\tau^{14}+53\tau^{12}-32\tau^{10}+396\tau^8\\
 &\hspace{35mm}{}-320\tau^6-336\tau^4+512\tau^2-192\bigr),\\
A_2={}&\tau^2(\tau^2+2)^2
 \bigl(12\tau^{14}+123\tau^{12}+240\tau^{10}+116\tau^8\\
 &\hspace{35mm}{}+960\tau^6-432\tau^4-256\tau^2+192\bigr),\\
A_3={}&-4\tau^7(\tau^2+2)^4
 (3\tau^6+14\tau^4+4\tau^2+40),\\
A_4={}&\tau^6(\tau^2+2)^6(3\tau^4+4\tau^2+12).
\end{align*}
For $\tau\ge2$ these coefficients have signs $+,-,+,-,+$ in
increasing degree.  For the only two non-obvious brackets, put
$q=\tau^2\ge4$ and group
\begin{align*}
&20q^7-32q^5+53q^6-320q^3+396q^4-336q^2+512q-192,\\
&12q^7+123q^6+240q^5+116q^4
  +(960q^3-432q^2-256q)+192;
\end{align*}
each displayed group is positive.  Thus $F_\tau(p)>0$ for $p<0$.

The discriminant $\Delta_\tau$ of $F_\tau$ is
\begin{align}
 \Delta_\tau={}&-5184\tau^{14}(\tau^2-2)^{12}(\tau^2+2)^{18}
 (2\tau^2-1)^2\notag\\
 &\quad\cdot(\tau^2-2\tau+2)^2(\tau^2+2\tau+2)^2H(\tau^2),
 \label{eq:discriminant}
\end{align}
where, on writing $z=q-4$,
\begin{align*}
H(q)={}&576z^{15}+42864z^{14}+1457244z^{13}+30074493z^{12}\\
&+421983272z^{11}+4268521080z^{10}+32174243296z^9\\
&+184016942896z^8+804672632576z^7+2686678018304z^6\\
&+6778308853760z^5+12646404948736z^4+16800286140416z^3\\
&+14875812689920z^2+7736848048128z+1732363472896.
\end{align*}
Hence \eqref{eq:discriminant} is strictly negative.  The quartic has exactly
two simple real roots; the sign pattern excludes negative roots.  Denote its
smaller positive root by $p_-(\tau)$.  It is continuous for $\tau\ge2$.

Set
\begin{align}
 U(p,\tau):={}&\bigl(3p^2\tau^{12}+14p^2\tau^{10}+24p^2\tau^8
 +48p^2\tau^6\notag\\
&+112p^2\tau^4+96p^2\tau^2-20p\tau^9-32p\tau^7-96p\tau^5\notag\\
&-128p\tau^3+192p\tau+8\tau^8-8\tau^6+112\tau^4
 -160\tau^2+64\bigr)\notag\\
&\bigm/\bigl(12\tau^2(\tau^2-2)^2(\tau^2+2)\bigr).
\label{eq:Uselector}
\end{align}
At $\tau=2$ one has
\begin{align*}
 F_2(2/5)&=\frac{39721984}{625}>0,&
 F_2(4/9)&=-35840<0,\\
 U(p,2)-2&=\frac{594p^2-282p+7}{18}.&&
\end{align*}
It follows that $2/5<p_-(2)<4/9$ and $U(p_-(2),2)<2$.

For $\tau^2\ge14$, direct substitution gives
\begin{align*}
 F_\tau(1/\tau)={}&3w^9+373w^8+20472w^7+651304w^6\\
 &+13240608w^5+178399616w^4+1593137664w^3\\
 &+9091884032w^2+30082498560w+43954733056>0,
\end{align*}
where $w=\tau^2-14$, while
\begin{align*}
 F_\tau(2/\tau)={}&-4(\tau^2-14)(3\tau^4+4\tau^2+4)\\
 &\quad\cdot(3\tau^{10}+24\tau^8+36\tau^6+104\tau^4\\
 &\hspace{42mm}{}-32\tau^2+32)<0
\end{align*}
when $\tau^2>14$.  Consequently
$1/\tau<p_-(\tau)<2/\tau$ for large $\tau$.  Substitution
$p=y/\tau$ gives, after division by $\tau^{18}$,
\[
       \tau^{-18}F_\tau(y/\tau)\longrightarrow3y^2(y-2)^2
       \quad\text{uniformly for }1\le y\le2.
\]
Thus $\tau p_-(\tau)\to2$, and \eqref{eq:Uselector} gives
\[
             \frac{U(p_-(\tau),\tau)}{\tau^2}\longrightarrow1.
\]
The continuous function $\tau\mapsto U(p_-(\tau),\tau)$ therefore starts
below $2$ and tends to infinity.  Its image contains $[2,\infty)$.

Fix the prescribed $u\ge2$ and choose $\tau\ge2$ and
$p=p_-(\tau)$ with $U(p,\tau)=u$.  Put $r=\sqrt p$, $\sigma=-\tau r$, and
\begin{align*}
b={}&-(\tau^2+2)p,\\
x={}&\frac{\tau r\{p(\tau^2+2)^2-4\tau\}}{2(\tau^2-2)},\\
m={}&\frac{r\{-p\tau(\tau^2+2)^2+8(\tau^2-1)\}}
              {2\tau(\tau^2-2)},\\
n={}&\frac{r\{p(\tau^2+2)^2-2\tau^3\}}{2(\tau^2-2)},\\
c={}&\frac{p\{-p\tau(\tau^2+2)^2-8+4\tau^2\}+6\tau u}{2\tau}.
\end{align*}
Exact coefficient comparison shows that \eqref{eq:twosos} is equivalent to
\begin{align*}
&b+2r^2+\sigma^2=0,\qquad
c-3u-x\sigma+2mr=0,\\
&b+2xr-m\sigma+\sigma n=0,\qquad
b^2+2x\sigma+4rn=0,\\
&bc+x^2+m^2+\sigma m+\sigma n+n^2=0,\\
&b^2+2x(m+n)+4r\sigma=0,\qquad
c^2-u^2+4x(m-n)=0.
\end{align*}
The displayed parameters solve the first four and the sixth equations.
The fifth is precisely $U(p,\tau)=u$; after this substitution the residual
in the seventh is
\[
 \frac{2F_\tau(p)}{9\tau^4(\tau^2-2)^4(\tau^2+2)^2}=0.
\]
This proves the identity.
\end{proof}

The preceding selector covers $u\ge2$.  The complementary interval admits
one matrix-polynomial certificate.  We write
\[
 \mathbf v(T)=(I,T,T^2,T^3,T^4)^\mathsf T,
 \qquad \mathbf w(T)=(I,T,T^2,T^3)^\mathsf T,
\]
and, for $d\ge1$,
\[
 V_{2,d}(u)=\begin{pmatrix}I_d\\uI_d\\u^2I_d\end{pmatrix},
 \qquad
 V_{1,d}(u)=\begin{pmatrix}I_d\\uI_d\end{pmatrix}.
\]

\begin{lemma}[Exact interval certificate]\label{lem:intervalcert}
There are Hermitian matrices
\begin{align*}
 S_0&\in M_{15}(\mathbb Q(\sqrt2)),&
 R_0&\in M_{10}(\mathbb Q(\sqrt2)),\\
 S_P&\in M_{12}(\mathbb Q(\sqrt2,i)),&
 R_P&\in M_8(\mathbb Q(\sqrt2,i)).
\end{align*}
which are positive semidefinite and have respective ranks $13,8,10,7$.
Put
\begin{align*}
 G_0(u)&=V_{2,5}(u)^*S_0V_{2,5}(u)
       +u(2-u)V_{1,5}(u)^*R_0V_{1,5}(u),\\
 G_P(u)&=V_{2,4}(u)^*S_PV_{2,4}(u)
       +u(2-u)V_{1,4}(u)^*R_PV_{1,4}(u),\\
 G_Q(u)&=\overline{G_P(u)}.
\end{align*}
For $0\le u\le2$, every $T=X+iY$, and
\[
 P=X+Y,\quad Q=X-Y,\quad Z=T^2,\quad
 D=Z^2+uZ+I,\quad N=Z^2-uZ+I,
\]
one has the exact hereditary identity
\begin{align}
 2D^*D-N^*N={}&\mathbf v(T)^*G_0(u)\mathbf v(T)
  +\mathbf w(T)^*G_P(u)P\mathbf w(T)\notag\\
 &+\mathbf w(T)^*G_Q(u)Q\mathbf w(T).
 \label{eq:interval-sos}
\end{align}
\end{lemma}

\begin{proof}
The four matrices are given entry by entry in the ancillary file
\texttt{ancillary/symmetric\_u\_exact\_certificate.py}; every entry is an
exact element of the displayed number field.  We record the finite audit so
that the certificate is independently checkable.  Put $r=1/\sqrt2$ and
\begin{align*}
 v_1&=(1,r,0,-r,-1)^\mathsf T,&
 v_2&=(0,r,1,r,0)^\mathsf T,\\
 p&=(-r(1+i),-i,r(1-i),1)^\mathsf T.
\end{align*}
The prescribed kernels are
\begin{center}
\begin{tabular}{cclc}
\toprule
matrix&order&kernel basis&rank\\
\midrule
$S_0$&15&$v_1\oplus0_{10},\ v_2\oplus0_{10}$&13\\
$R_0$&10&$v_1\oplus0_5,\ v_2\oplus0_5$&8\\
$S_P$&12&$p\oplus0_8,\ \bar p\oplus0_8$&10\\
$R_P$&8 &$p\oplus0_4$&7\\
\bottomrule
\end{tabular}
\end{center}
On a full-rank principal compression of each matrix, exact symmetric
elimination gives respectively $13,8,10,7$ positive $LDL^*$ pivots in
$\mathbb Q(\sqrt2)$.  Positivity of an element $a+b\sqrt2$ is decided using
only rational comparisons (for opposite signs, compare $a^2$ with $2b^2$).
Thus the four matrices are positive semidefinite with the asserted ranks.
The ancillary verifier performs precisely these operations in exact SymPy
arithmetic; it uses no floating-point test.

It remains to check the identity.  Put
\[
 d_0=(1,0,0,0,1)^\mathsf T,
 \qquad d_1=(0,0,1,0,0)^\mathsf T.
\]
Coefficient comparison in powers of $u$ gives the five target matrices
\[
 d_0d_0^*,\qquad
 3(d_0d_1^*+d_1d_0^*),\qquad
 d_1d_1^*,\qquad0,\qquad0.
\]
Expanding the three terms on the right of
\eqref{eq:interval-sos} gives exactly the same five matrices.  This is a
finite equality in $\mathbb Q(\sqrt2,i)$, checked entrywise by the ancillary
verifier.  Hence \eqref{eq:interval-sos} is an exact polynomial identity.
Finally the Markov--Lukacs forms defining $G_0,G_P,G_Q$ are positive
semidefinite for $0\le u\le2$, proving the lemma.
\end{proof}

\begin{proof}[Proof of \cref{thm:twozero}]
Put $Z=T^2$, $D=Z^2+uZ+I$, and $N=Z^2-uZ+I$.  Spectral inclusion for the
numerical range and spectral mapping give
$\Sp(Z)\subset\{z:\operatorname{Re}z\ge0\}$.  Both zeros of
$z^2+uz+1$ lie in the open left half-plane, so $D$ is invertible.
Sector containment gives $P=X+Y\ge0$ and $Q=X-Y\ge0$.  If $u\ge2$, apply
\cref{lem:selector}; if $0<u\le2$, apply \cref{lem:intervalcert}.  In either
case the relevant identity has a positive right-hand side.  Congruence by
$D^{-1}$ yields
\[
                 2I-(ND^{-1})^*(ND^{-1})\ge0,
\]
which is \eqref{eq:twozero}.

For positive real half-plane zeros $s,t$, replace $T$ by
$T/(st)^{1/4}$ and take $u=(s+t)/\sqrt{st}\ge2$.  For a conjugate pair
$\zeta=a+ib$, $\bar\zeta$, with $a>0$, replace $T$ by
$T/\sqrt{|\zeta|}$ and take $u=2a/|\zeta|\in(0,2]$.  Indeed, with
$b_\zeta(z)=(z-\zeta)/(z+\bar\zeta)$,
\[
 b_\zeta(z)b_{\bar\zeta}(z)
   =\frac{z^2-2az+|\zeta|^2}{z^2+2az+|\zeta|^2}.
\]
This proves both stated Blaschke-product consequences.
\end{proof}

\subsection{The complete imaginary post-automorphism diameter}

Let $J_4$ denote the $4\times4$ reversal matrix.  The following exact
matrix certificate is the core of \cref{thm:imagdiam}.

\begin{lemma}[Exact imaginary-diameter certificate]
\label{lem:imagcert}
Put
\[
       a=\frac12\sqrt{1+\sqrt2},\qquad \mathbb K=\mathbb Q(i,a).
\]
There are positive semidefinite Hermitian matrices
\begin{align*}
 S_0&\in M_{15}(\mathbb K),& R_0&\in M_{10}(\mathbb K),\\
 S_P&\in M_{12}(\mathbb K),& R_P&\in M_8(\mathbb K),
\end{align*}
of respective ranks $11,8,8,7$.  For $0\le t\le1$, set
\begin{align*}
 G_0(t)&=V_{2,5}(t)^*S_0V_{2,5}(t)
       +t(1-t)V_{1,5}(t)^*R_0V_{1,5}(t),\\
 G_P(t)&=V_{2,4}(t)^*S_PV_{2,4}(t)
       +t(1-t)V_{1,4}(t)^*R_PV_{1,4}(t),\\
 G_Q(t)&=J_4G_P(t)J_4.
\end{align*}
If $T=X+iY$, $P=X+Y$, $Q=X-Y$, and
\begin{align*}
 D_t&=(T^2+I)^2+it(T^2-I)^2,\\
 N_t&=(T^2-I)^2-it(T^2+I)^2,
\end{align*}
then the hereditary identity
\begin{align}
 2D_t^*D_t-N_t^*N_t={}&
 \mathbf v(T)^*G_0(t)\mathbf v(T)
 +\mathbf w(T)^*G_P(t)P\mathbf w(T)\notag\\
 &+\mathbf w(T)^*G_Q(t)Q\mathbf w(T)
 \label{eq:imag-sos}
\end{align}
holds exactly.
\end{lemma}

\begin{proof}
We give the finite audit because the matrices are most compactly recorded by
exact affine coordinates.  Their machine-readable data are in
\texttt{ancillary/imaginary\_interval\_certificate.py}, and the independent
reconstruction is
\texttt{ancillary/verify\_imaginary\_interval.py}.

Write
\[
 \mathbf d=(1,0,2,0,1)^\mathsf T,
 \qquad \mathbf n=(1,0,-2,0,1)^\mathsf T.
\]
The coefficient matrices of $2D_t^*D_t-N_t^*N_t$ in degrees $0,1,2$ are
\begin{equation}\label{eq:imagtargets}
 2\bar{\mathbf d}\mathbf d^\mathsf T-
       \bar{\mathbf n}\mathbf n^\mathsf T,\quad
 i(\bar{\mathbf d}\mathbf n^\mathsf T-
       \bar{\mathbf n}\mathbf d^\mathsf T),\quad
 -\bar{\mathbf d}\mathbf d^\mathsf T+
       2\bar{\mathbf n}\mathbf n^\mathsf T.
\end{equation}
Starting with Hermitian indeterminates of orders $15,10,12,8$, expand the
two Markov--Lukacs forms in the statement, use $G_Q=J_4G_PJ_4$, and equate
all hereditary coefficients through degree four to
\eqref{eq:imagtargets} and two zero matrices.  At $t=1$ impose the natural
endpoint face
\begin{align*}
 V_{2,5}(1)^*S_0V_{2,5}(1)
   &=\bar{\mathbf d}_1\mathbf d_1^\mathsf T,
 &V_{2,4}(1)^*S_PV_{2,4}(1)&=0,\\
 \mathbf d_1&=(1+i,0,2-2i,0,1+i)^\mathsf T.&&
\end{align*}
Finally impose the following kernel dimensions:
\begin{center}
\begin{tabular}{cclc}
\toprule
matrix&order&kernel vectors&rank\\
\midrule
$S_0$&15&$(q,q,q)$ for four independent $\mathbf d_1^\mathsf Tq=0$&11\\
$R_0$&10&$(q_1,q_1),(q_2,q_2)$&8\\
$S_P$&12&$(e_j,e_j,e_j)$, $0\le j\le3$&8\\
$R_P$&8&$(p,p)$&7\\
\bottomrule
\end{tabular}
\end{center}
Here $e_j$ are the coordinate vectors and, with
\[
 \eta=\frac12\left(\sqrt{1+\sqrt2}
              +i\sqrt{\sqrt2-1}\right),\quad
 \rho=\sqrt2-1,\quad
 \delta=\sqrt{\frac{\sqrt2-1}{2}},
\]
one may take
\begin{align*}
q_1&=(i,\eta,1,\eta,i)^\mathsf T,\\
q_2&=(1,\bar\eta/\sqrt2,0,-\bar\eta/\sqrt2,-1)^\mathsf T,\\
p&=(-\rho\delta(1+i),-i\rho,\delta(1-i),1)^\mathsf T.
\end{align*}
For the first row of the table the verifier uses
\[
 (1,0,0,0,-1)^\mathsf T,\quad(2i,0,1,0,0)^\mathsf T,\quad
 e_1,\quad e_3.
\]
All these entries lie in $\mathbb K$, since
$16a^4-8a^2-1=0$ and the remaining real radicals are rational functions of
$a$.

After real and imaginary parts are separated, the construction is a sparse
$438\times534$ exact affine system of rank $328$.  The ancillary certificate
lists its $205$ free indices and integer numerators, with common denominator
$10^4$; exact row reduction determines every pivot coordinate.  Substitution
gives zero in all $438$ equations, which proves \eqref{eq:imag-sos}.

It remains only to certify positivity.  The displayed kernels are
independent.  On full-rank principal compressions the verifier performs
exact $LDL^*$ elimination.  The selected orders and the smallest pivots are
shown below; the decimal column is only a readable display of an exact sign
calculation.
\begin{center}
\small
\begin{tabular}{cclc}
\toprule
matrix&order&principal indices (zero based)&smallest pivot\\
\midrule
$S_0$&11&$2,12,5,9,11,13,7,1,3,0,4$&$0.127165\ldots$\\
$R_0$&8 &$2,1,6,8,7,3,0,4$&$0.146820\ldots$\\
$S_P$&8 &$2,1,4,6,9,7,8,11$&$0.139901\ldots$\\
$R_P$&7 &$1,2,5,6,4,0,7$&$0.166670\ldots$\\
\bottomrule
\end{tabular}
\end{center}
No floating-point sign decision is used.  Every real pivot is reduced to a
cubic polynomial in $a$.  Rational endpoints first isolate the positive
root of $16x^4-8x^2-1$ (the polynomial is strictly increasing for
$x>1/2$); rational interval Horner evaluation then gives a strictly positive
lower endpoint for every pivot.  Thus the four matrices are positive
semidefinite with exactly the ranks in the statement, completing the exact
audit.  If $\Pi_0$ is the orthogonal projection onto $\ker S_0$, the same
exact elimination on the order-$11$ compression also gives the quantitative
margin
\begin{equation}\label{eq:S0margin}
                    S_0\ge\frac1{10}(I-\Pi_0).
\end{equation}
The smallest pivot in this additional audit is $0.069496\ldots$; once again
its strictly positive rational interval, rather than the decimal display,
is what proves the assertion.
\end{proof}

\begin{proof}[Proof of \cref{thm:imagdiam}]
First suppose $0\le t<1$.  Sector containment gives $P,Q\ge0$, so
\cref{lem:imagcert} makes the right side of \eqref{eq:imag-sos} positive.
Put $B=T^2$ and $w=(B-I)(B+I)^{-1}$.  Spectral inclusion puts $\Sp(B)$ in
the closed right half-plane, hence every scalar spectral value of $w$
has modulus at most one.  Since $t<1$, spectral mapping shows that
$I+itw^2$ is invertible; equivalently, $D_t$ is invertible.  Congruence of
\eqref{eq:imag-sos} by $D_t^{-1}$ gives
\[
                  2I-(N_tD_t^{-1})^*(N_tD_t^{-1})\ge0,
\]
and proves \eqref{eq:imagdiam} for $0\le t<1$.  At $t=1$ one has
$N_1=-iD_1$, so the reduced rational function is the constant $-i$.

For $-1\le t<0$, apply the result for $-t$ to $T^*$.  Directly from
\eqref{eq:hc},
\[
                       h_c(T^2)^*=h_{\bar c}((T^*)^2),
\]
and $W(T^*)\subset\cS_{\pi/4}$.  The endpoint $t=-1$ similarly reduces to
the constant $i$.  This proves the full diameter.
\end{proof}

\begin{proof}[Proof of \cref{thm:cusp}]
It suffices by the adjoint symmetry in the preceding proof to take
$c=x+it$ with $0\le t\le1$ and $x\in\mathbb R$.  Let
\[
          E=\ker\mathbf d_1^\mathsf T\subset\mathbb C^5.
\]
The kernel of $S_0$ in \cref{lem:imagcert} is
$\{(q,q,q):q\in E\}$.  For $\xi\in\mathbb C^5$ put
$y=(\xi,t\xi,t^2\xi)^\mathsf T$.  Orthogonally decomposing $\xi$ into
$E\oplus E^\perp$ gives
\begin{align*}
 \dist(y,\ker S_0)^2
 &\ge \left(1+t^2+t^4-\frac{(1+t+t^2)^2}{3}\right)\norm{\xi}^2\\
 &=\frac23(1-t)^2(1+t+t^2)\norm{\xi}^2
 \ge\frac23(1-t)^2\norm{\xi}^2.
\end{align*}
Combining this with \eqref{eq:S0margin} and $R_0\ge0$ yields
\begin{equation}\label{eq:G0cusp}
                         G_0(t)\ge\frac1{15}(1-t)^2I.
\end{equation}

Put
\[
 \mathbf d=(1,0,2,0,1)^\mathsf T,\qquad
 \mathbf n=(1,0,-2,0,1)^\mathsf T,
\]
and put
\[
 D_c=(T^2+I)^2+c(T^2-I)^2,\qquad
 N_c=(T^2-I)^2+\bar c(T^2+I)^2.
\]
Let $H(c)$ be the coefficient matrix of $2D_c^*D_c-N_c^*N_c$ in the
basis $\mathbf v(T)$.  Direct expansion gives
\[
 H(x+it)-H(it)
 =x(\bar{\mathbf d}\mathbf n^\mathsf T
       +\bar{\mathbf n}\mathbf d^\mathsf T)
  +x^2(2\bar{\mathbf n}\mathbf n^\mathsf T
       -\bar{\mathbf d}\mathbf d^\mathsf T).
\]
Because $\norm{\mathbf d}=\norm{\mathbf n}=\sqrt6$,
\begin{equation}\label{eq:cusppert}
                \norm{H(x+it)-H(it)}\le12|x|+18x^2.
\end{equation}
If $|x|\le(1-t)^2/200$, then \cref{eq:G0cusp,eq:cusppert} give
\begin{align*}
G_0(t)+H(x+it)-H(it)
&\ge\left(\frac1{15}-\frac{12}{200}
                 -\frac{18}{200^2}\right)(1-t)^2I\\
&=\frac{373}{60000}(1-t)^2I.
\end{align*}
For $t<1$ this is positive.  Replace $G_0(t)$ by this perturbed ordinary
Gram matrix in \eqref{eq:imag-sos}; the two localized Grams are unchanged.
Condition \eqref{eq:cusp} also gives $|c|<1$ away from the tips, so the same
spectral-mapping argument makes the denominator in \eqref{eq:cuspbound}
invertible.  Congruence proves the estimate.  At $t=1$ condition
\eqref{eq:cusp} forces $x=0$, which is the already treated constant endpoint.
\end{proof}

\subsection{A full complex parameter disk}

We prove \cref{thm:centraldisk} by one bivariate matrix Putinar
certificate.  Order the monomials of total degree at most $k$ as
\[
 \mathbf m_k(x,y)=(x^jy^{d-j}:0\le d\le k,\ 0\le j\le d)^\mathsf T
\]
and put $\mathcal V_{k,m}(x,y)=\mathbf m_k(x,y)\otimes I_m$.

\begin{lemma}[Exact complex-disk certificate]\label{lem:complexdiskcert}
There are positive definite rational Hermitian matrices
\begin{align*}
 S_0&\in M_{50}(\mathbb Q(i)),&R_0&\in M_{30}(\mathbb Q(i)),\\
 S_P&\in M_{40}(\mathbb Q(i)),&R_P&\in M_{24}(\mathbb Q(i))
\end{align*}
with the following property.  If $c=x+iy$, $x,y\in\mathbb R$, set
\begin{align*}
 G_0(x,y)&=\mathcal V_{3,5}(x,y)^*S_0\mathcal V_{3,5}(x,y)\\
 &\quad+\left(\frac{361}{400}-x^2-y^2\right)
          \mathcal V_{2,5}(x,y)^*R_0\mathcal V_{2,5}(x,y),\\
 G_P(x,y)&=\mathcal V_{3,4}(x,y)^*S_P\mathcal V_{3,4}(x,y)\\
 &\quad+\left(\frac{361}{400}-x^2-y^2\right)
          \mathcal V_{2,4}(x,y)^*R_P\mathcal V_{2,4}(x,y),\\
 G_Q(x,y)&=\overline{G_P(x,-y)}.
\end{align*}
For $T=X+iY$, $P=X+Y$, $Q=X-Y$, and
\begin{align*}
 D_c&=(T^2+I)^2+c(T^2-I)^2,\\
 N_c&=(T^2-I)^2+\bar c(T^2+I)^2,
\end{align*}
one has the exact hereditary identity
\begin{align}
2D_c^*D_c-N_c^*N_c={}&
 \mathbf v(T)^*G_0(x,y)\mathbf v(T)
 +\mathbf w(T)^*G_P(x,y)P\mathbf w(T)\notag\\
 &+\mathbf w(T)^*G_Q(x,y)Q\mathbf w(T).
\label{eq:complex-disk-sos}
\end{align}
\end{lemma}

\begin{proof}
Put
\[
 \mathbf d=(1,0,2,0,1)^\mathsf T,
 \qquad \mathbf n=(1,0,-2,0,1)^\mathsf T.
\]
The coefficient matrices of $2D_c^*D_c-N_c^*N_c$ at the parameter
monomials $1,x,y,x^2,xy,y^2$ are, respectively,
\begin{align}\label{eq:complex-disk-targets}
 &2\bar{\mathbf d}\mathbf d^\mathsf T
      -\bar{\mathbf n}\mathbf n^\mathsf T,\qquad
 \bar{\mathbf d}\mathbf n^\mathsf T
      +\bar{\mathbf n}\mathbf d^\mathsf T,\qquad
 i(\bar{\mathbf d}\mathbf n^\mathsf T
      -\bar{\mathbf n}\mathbf d^\mathsf T),\notag\\
 &2\bar{\mathbf n}\mathbf n^\mathsf T
      -\bar{\mathbf d}\mathbf d^\mathsf T,\qquad
 0,\qquad
 2\bar{\mathbf n}\mathbf n^\mathsf T
      -\bar{\mathbf d}\mathbf d^\mathsf T.
\end{align}
Expand the four matrix Putinar terms in the statement and equate their
hereditary output to \eqref{eq:complex-disk-targets}, with
$G_Q(x,y)=\overline{G_P(x,-y)}$.  The integer seed file stores the Hermitian
coordinates of the four constant Gram matrices with common denominator
$10^{12}$.  Rounding alone need not preserve the affine equations, so the
verifier removes its residual exactly.  For a parameter monomial $x^jy^k$,
it applies a rational right inverse of the fixed finite-dimensional map
\[
 (H_0,H_P)\longmapsto
 \operatorname{Her}\bigl(H_0,H_P,(-1)^k\overline{H_P}\bigr)
\]
and lifts the resulting correction into $S_0$ and $S_P$.  Every monomial of
degree at most six is a product of two monomials of degree at most three, so
this lift is exact.  A fresh expansion gives precisely the six matrices in
\eqref{eq:complex-disk-targets} and zero at the other $22$ parameter
monomials.

We also record the exact positivity audit.  For each corrected matrix $A$,
the verifier constructs rational matrices $L,M$ and puts
\[
 R=A-LL^*,\qquad E=I-ML.
\]
Set
\[
 \epsilon=\sqrt{\norm{E}_1\norm{E}_\infty},\quad
 \mu^2=\norm{M}_1\norm{M}_\infty,\quad
 \ell=\frac{(1-\epsilon)^2}{\mu^2},\quad
 \rho=\sqrt{\norm{R}_1\norm{R}_\infty}.
\]
The square roots are replaced in the computation by rational upper bounds.
Since $ML=I-E$, the inequalities $\epsilon<1$ and $\ell>\rho$ imply
$\lambda_{\min}(A)\ge\ell-\rho>0$.  The exact rational comparisons, with
decimal displays, are
\begin{center}
\begin{tabular}{c|rrrr}
\toprule
matrix&$S_0$&$R_0$&$S_P$&$R_P$\\
\midrule
order&50&30&40&24\\
$\ell$&$9.27\mathord\cdot10^{-8}$&$2.10\mathord\cdot10^{-6}$
       &$8.12\mathord\cdot10^{-8}$&$1.15\mathord\cdot10^{-5}$\\
$\rho$&$1.63\mathord\cdot10^{-11}$&$3.70\mathord\cdot10^{-12}$
       &$8.29\mathord\cdot10^{-12}$&$4.25\mathord\cdot10^{-12}$\\
\bottomrule
\end{tabular}
\end{center}
The decimals only display certified rational bounds.  The independent
verifier reads no floating-point discovery Gram: it reconstructs the exact
affine point from the integer seed, checks all $28$ hereditary coefficient
matrices, and repeats the four positivity comparisons.  This proves the
lemma.
\end{proof}

\begin{proof}[Proof of \cref{thm:centraldisk}]
Write $c=x+iy$.  Under \eqref{eq:cdisk}, each of the three matrices
$G_0(x,y),G_P(x,y),G_Q(x,y)$ in \cref{lem:complexdiskcert} is positive
semidefinite.  Sector containment gives $P,Q\ge0$, so the right side of
\eqref{eq:complex-disk-sos} is positive.

Put $B=T^2$ and $w=(B-I)(B+I)^{-1}$.  Spectral inclusion puts
$\Sp(B)$ in the closed right half-plane, hence $\Sp(w)\subset\overline{\mathbb
D}$.  Since $|c|\le19/20<1$, the operator $I+cw^2$ is invertible; equivalently,
$D_c$ is invertible.  Congruence of \eqref{eq:complex-disk-sos} by
$D_c^{-1}$ proves \eqref{eq:centraldisk}.  At $c=i/2$, direct factorization
of the numerator gives
\[
 (1-i/2)\left(z-(1+2i)\right)
          \left(z-\frac{1-2i}{5}\right),
\]
which proves \eqref{eq:nonsymzeros}.
\end{proof}

\subsection{A boundary-reaching phase arc}

We now prove \cref{thm:phasewedge}.  The proof has three finite exact parts:
a rational seed certificate identifying the forced faces, a rational
interpolation audit showing that $72$ selected coefficient equations imply
the full hereditary identity, and six matrix Bernstein certificates giving
strict rational sections in two complementary row charts.
Put $R=19/20$ and retain the vectors
\[
 \mathbf d=(1,0,2,0,1)^\mathsf T,
 \qquad \mathbf n=(1,0,-2,0,1)^\mathsf T.
\]

\begin{lemma}[Exact radial seed and continuation]\label{lem:radialcontinue}
For $1/10\le t\le18/25$ and $R\le\rho\le1$, there are positive semidefinite
Hermitian matrices $G_0(\rho,t)\in M_5(\C)$ and
$G_P(\rho,t),G_Q(\rho,t)\in M_4(\C)$ for which
\begin{align}
 2D_{\rho,t}^*D_{\rho,t}-N_{\rho,t}^*N_{\rho,t}
  ={}&\mathbf v(T)^*G_0(\rho,t)\mathbf v(T)\notag\\
 &+\mathbf w(T)^*G_P(\rho,t)P\mathbf w(T)
  +\mathbf w(T)^*G_Q(\rho,t)Q\mathbf w(T),
 \label{eq:radialher}
\end{align}
where
\begin{align*}
 D_{\rho,t}&=(T^2+I)^2+\rho\zeta_t(T^2-I)^2,\\
 N_{\rho,t}&=(T^2-I)^2+\rho\bar\zeta_t(T^2+I)^2.
\end{align*}
The matrices may be chosen continuously in $t$.
\end{lemma}

\begin{proof}
For fixed $t$, the tangential Grams have the form
\begin{equation}\label{eq:radialedge}
 G_P(\rho,t)=(1-\rho^2)\widetilde G_P(\rho,t),
 \qquad
 G_Q(\rho,t)=(1-\rho^2)\widetilde G_Q(\rho,t),
\end{equation}
where each tilde Gram is the affine interpolation of positive endpoint
Grams at $\rho=R$ and $\rho=1$.  For the ordinary Gram put
\[
 \mathcal R_\rho=\binom{I_5}{\rho I_5}
\]
and use the odd matrix Markov--Lukacs form
\begin{equation}\label{eq:radialml}
 G_0(\rho,t)=(\rho-R)\mathcal R_\rho^*S(t)\mathcal R_\rho
             +(1-\rho)\mathcal R_\rho^*U(t)\mathcal R_\rho.
\end{equation}
Thus positivity on $[R,1]$ follows once the two endpoint edge Grams and
$S(t),U(t)$ are positive on their forced supports.

Those supports are rational in $t$.  Define
\begin{equation}\label{eq:kpt}
 k_P(t)=\bigl(-2t^3(1+i),-2it^2,t(1-i),1\bigr)^\mathsf T,
 \qquad k_Q(t)=J_4k_P(t).
\end{equation}
The endpoint edge supports are $k_P(t)^\perp$ and $k_Q(t)^\perp$.
Writing $e_t=\mathbf d+\zeta_t\mathbf n$, a six-dimensional support for
$S(t)$ is
\begin{equation}\label{eq:Stsupport}
 \{(x,-x):x\in\C^5\}
       +\operatorname{span}\{(\bar e_t,0)\}.
\end{equation}
For the other Markov--Lukacs multiplier set
\begin{equation}\label{eq:qt}
 q_t(z)=\bigl(z-t(1+i)\bigr)
        \left(z-\frac{1-i}{2t}\right)
\end{equation}
and let $\mathcal Q_t\in M_{5,3}(\C)$ be the coefficient-convolution map
$p\mapsto q_tp$ from $\C[z]_{\le2}$ to $\C[z]_{\le4}$.  An
eight-dimensional support for $U(t)$ is
\begin{equation}\label{eq:Utsupport}
 \{(x,-x):x\in\C^5\}
       +\{(\mathcal Q_ty,0):y\in\C^3\}.
\end{equation}

At $t_0=1/2$ one has $s_{t_0}=3/2$ and
$\zeta_{t_0}=(7+24i)/25$.  The ancillary rational data give six positive
definite reduced Grams on the supports above, of orders
\begin{equation}\label{eq:radialorders}
                 4,\ 3,\ 4,\ 3,\ 6,\ 8.
\end{equation}
Their smallest exact $LDL^*$ pivots have the decimal displays
\begin{equation}\label{eq:radialpivots}
 \begin{gathered}
 .0388996584,\quad .5674026743,\quad .2919816307,\\
 4.469610333,\quad .0196053190,\quad .0223098226.
 \end{gathered}
\end{equation}
Every number in \eqref{eq:radialpivots} displays a positive rational pivot.
The verifier re-expands all four coefficients in $\rho$ in
\eqref{eq:radialher} exactly.  It also checks that the exact rank-two
projector occurring in the $U$-face annihilates $\mathcal Q_{t_0}$.
Since both spaces have complementary dimensions two and three, this proves
\eqref{eq:Utsupport} without a numerical eigenspace assertion.

It remains to justify that the strict certificate continues with $t$.
Choose the elementary rational bases in
\eqref{eq:kpt}--\eqref{eq:Utsupport}, and record Hermitian matrices by real
coordinates.  Comparing the four coefficients in $\rho$ gives
\begin{equation}\label{eq:Atbt}
                  A(t)x=b(t),\qquad
 A(t)\in M_{144,150}(\mathbb Q(t)).
\end{equation}
A fixed set $\mathcal R$ of $72$ rows and $\mathcal C$ of $72$ columns,
selected at $t_0$, gives a square minor
$M(t)=A(t)_{\mathcal R,\mathcal C}$ with $M(t_0)\ne0$.  The exact identities
\begin{equation}\label{eq:rightrational}
 A(t)_{:,\mathcal C}M(t)^{-1}A(t)_{\mathcal R,:}=A(t),
 \qquad
 A(t)_{:,\mathcal C}M(t)^{-1}b(t)_{\mathcal R}=b(t)
\end{equation}
hold in $\mathbb Q(t)$.

For completeness, we describe the finite audit of
\eqref{eq:rightrational}.  Clear denominators independently in each column.
Among the $72$ selected columns, the resulting degree multiset is
\begin{equation}\label{eq:degreemultiset}
 0^{(34)},\ 2^{(6)},\ 3^{(6)},\ 4^{(10)},\
 5^{(4)},\ 6^{(3)},\ 8^{(9)}.
\end{equation}
An arbitrary added map or target column has degree at most eight.  Hence
every relevant bordered determinant has degree at most
\[
 34\cdot0+6\cdot2+6\cdot3+10\cdot4+4\cdot5
       +3\cdot6+9\cdot8+8=188.
\]
At each of the $189$ distinct rational points $t=1,\ldots,189$, exact
rational inversion of $M(t)$ gives both identities in
\eqref{eq:rightrational} entry by entry.  The bordered determinants therefore
have $189$ roots and vanish identically.  The ancillary continuation audit
performs the denominator, degree, inversion, and residual checks from
scratch; no floating-point rank decision enters this interpolation step.

We next give the interval certificates.  Let
\begin{equation}\label{eq:phaseintervals}
\begin{gathered}
 I_1=[1/10,1/5],\quad I_2=[1/5,3/10],\quad I_3=[3/10,2/5],\\
 I_4=[2/5,1/2],\quad I_5=[1/2,3/5],\quad I_6=[3/5,18/25].
\end{gathered}
\end{equation}
and write $u_j=(t-\inf I_j)/(\sup I_j-\inf I_j)$.  The first five
intervals use the row set $\mathcal R_0=\mathcal R$ above.  A second fixed
set $\mathcal R_1$ of $72$ rows, selected at $t=7/10$, is used on $I_6$;
all indices are recorded in the verifier.  Put
\[
                     d=(6,6,6,6,6,8),
 \qquad \beta_{j,k}(u)=\binom{d_j}{k}u^k(1-u)^{d_j-k}.
\]
For each $j$, the ancillary integer data specify a Bernstein section
\begin{equation}\label{eq:bernsteinsection}
             \widetilde x_j(u)=
                \sum_{k=0}^{d_j}\beta_{j,k}(u)X_{j,k}
                \quad (X_{j,k}\in\mathbb Q^{150}).
\end{equation}
As before, every control vector encodes six reduced Hermitian blocks in the
orders \eqref{eq:radialorders}.  Fraction-free $LDL^*$ proves that every
control block exceeds $m_jI$.  The same data give a polynomial matrix
$B_j(u)\in M_{150,72}(\mathbb Q[u])$.  With
\[
 A_j(t)=A(t)_{\mathcal R_0,:}\quad(j\le5),
 \qquad A_6(t)=A(t)_{\mathcal R_1,:},
\]
and the analogous notation $b_j$, exact Bernstein arithmetic proves
\begin{equation}\label{eq:intervalbounds}
 \begin{gathered}
 \norm{b_j-A_j\widetilde x_j}_\infty<\epsilon_j,
 \qquad \norm{B_j}_\infty<L_j,\\
 E_j:=I_{72}-A_jB_j,\qquad \norm{E_j}_\infty<\eta_j,
 \end{gathered}
\end{equation}
where the certified bounds are
\[
\begin{array}{c|c|c|c|c|c}
j&m_j&\epsilon_j&L_j&\eta_j&
 m_j-16L_j\epsilon_j/(1-\eta_j)\\ \hline
1&1/20&1/300000&403&2/25&919/34500\\
2&1/30&1/500000&201&1/190&10579/393750\\
3&1/40&1/500000&209&1/1000&18287/999000\\
4&1/51&1/1000000&247&3/1000&99431/6355875\\
5&1/60&1/2000000&302&1/90&47471/3337500\\
6&1/70&1/1000000&198&1/500&9703/873250
\end{array}
\]
Every entry in the last column is greater than $1/100$.

Here are the exact details behind these finite inequalities.  The common
positive denominator
\begin{equation}\label{eq:phasedenominator}
 \Delta(t)=20t^2(2t^2-2t+1)^2(2t^2+2t+1)^2
\end{equation}
clears every selected entry of $A$ and $b$.  After substituting $t$ as an
affine function of $u_j$, its Bernstein controls are strictly positive.
The entries of $\Delta(b_j-A_j\widetilde x_j)$ and
$\Delta(I-A_jB_j)$ are polynomials of explicitly bounded degree.  The
verifier expands every scalar entry in the Bernstein basis over $\mathbb Q$,
takes the maximum absolute row sum of the controls, and divides by the
smallest control of $\Delta$.  This gives \eqref{eq:intervalbounds} without
interval sampling.  It also applies exact $LDL^*$ to all
$5\cdot7\cdot6+9\cdot6=264$ rational Gram controls.  The smallest pivot
after subtracting the stated $m_jI$ has decimal display
\[
                         0.0002415456052637.
\]
The stored candidates and right-inverse controls have common denominator
$10^{10}$; only these integers, not the discovery SDP, are read by the
verifier.

Since $\eta_j<1$, the Neumann lemma makes $I-E_j$ invertible and
\[
                         B_j(I-E_j)^{-1}
\]
is an exact right inverse of $A_j$.  If
$r_j=b_j-A_j\widetilde x_j$, define
\begin{equation}\label{eq:intervalcorrection}
 x_j=\widetilde x_j+B_j(I-E_j)^{-1}r_j.
\end{equation}
Then $A_jx_j=b_j$ and
\begin{equation}\label{eq:coordinatecorrection}
 \norm{x_j-\widetilde x_j}_\infty
             <\frac{L_j\epsilon_j}{1-\eta_j}.
\end{equation}
The rational identities \eqref{eq:rightrational} imply that the augmented
matrix $[A(t)\ b(t)]$ has rank at most $72$ for every $t>0$, by vanishing of
its $73$-minors.  On $I_6$, \eqref{eq:intervalbounds} shows that the second
row chart has rank $72$ as well.  Thus solving either selected row system
solves all $144$ coefficient equations, including at points where the base
minor is singular.

Finally, if every real coordinate of a Hermitian block of order at most
eight changes by at most $q$, its maximum absolute row sum, and hence its
operator norm, changes by less than $16q$.  Equations
\eqref{eq:coordinatecorrection} and the last column of the table therefore
show that every corrected reduced block is larger than $I/100$.  Equations
\eqref{eq:radialedge} and \eqref{eq:radialml} finish the proof.
\end{proof}

\begin{proof}[Proof of \cref{thm:phasewedge}]
For $0\le\rho\le R$, the assertion is \cref{thm:centraldisk}.  For
$R\le\rho<1$, sector containment gives $P,Q\ge0$, so
\cref{lem:radialcontinue} implies
$2D_{\rho,t}^*D_{\rho,t}-N_{\rho,t}^*N_{\rho,t}\ge0$.
Moreover $I+\rho\zeta_tw^2$ is invertible.  Congruence by
$D_{\rho,t}^{-1}$ proves \eqref{eq:phasewedgebound}.
At $\rho=1$ the scalar rational function cancels to the constant
$\bar\zeta_t$, so the endpoint follows directly.

Finally $W(T^*)=\overline{W(T)}\subset\cS_{\pi/4}$.  Applying the result to
$T^*$ and taking adjoints, using commutativity of the two polynomial factors,
replaces $c$ by $\bar c$.  Since $s'_{1/2}=-10$, the phases in
\eqref{eq:zetat} trace a nondegenerate arc.  More precisely,
$s_t=1/(2t^2)-2t^2$ is strictly decreasing, and
\[
                 \zeta_t=\frac{2+is_t}{2-is_t}.
\]
Thus $0<t\le1/\sqrt2$ traces the upper semicircle from $-1$ to $1$.
At $t=1/10$ the omitted upper angle is
$2\arctan(100/2499)$; reflection doubles it and gives the stated total
omitted length.  This proves all assertions.
\end{proof}

\begin{proof}[Proof of \cref{prop:andocentral}]
Put
\[
 R=(T-I)(T+I)^{-1},\qquad
 u=\frac{R+iI}{\sqrt2},\qquad v=\frac{R-iI}{\sqrt2}.
\]
Congruence by $T+I$ and the two sector inequalities
$\Re(e^{\pm i\pi/4}T)\ge0$ give
\begin{equation}\label{eq:Rdefects}
 I-R^*R+i(R-R^*)\ge0,
 \qquad I-R^*R-i(R-R^*)\ge0.
\end{equation}
Consequently $u$ and $v$ are commuting contractions and
$u-v=i\sqrt2I$.  This affine relation also makes both operators invertible.
Indeed, $\norm{ux}\ge(\sqrt2-1)\norm{x}$ follows from
$u=v+i\sqrt2I$, and the same estimate for $u^*$ makes the range of $u$
dense; the argument for $v$ is identical.

Set $a=\sqrt2$ and define
\begin{equation}\label{eq:CDdef}
 C=(u-iaI)^{-1}-iaI,
 \qquad D=(v+iaI)^{-1}+iaI.
\end{equation}
The scalar linear-fractional maps in \eqref{eq:CDdef} map $\mathbb D$ into
itself.  For example,
\[
 \left|\frac1{z-ia}-ia\right|\le1\quad(|z|\le1)
\]
follows after squaring from
$|z-ia|^2-|1+iaz|^2=1-|z|^2$.
Von Neumann's inequality therefore shows that $C$ and $D$ are commuting
contractions.  Since $u-iaI=v$ and $v+iaI=u$, elementary algebra gives,
with $P=CD$,
\begin{equation}\label{eq:CDrelations}
 \frac{C+D}{\sqrt2}=w,
 \qquad C-D=-i\sqrt2P,
 \qquad w^2=P(2I-P).
\end{equation}
For the first identity we used
$uv=(I+R^2)/2$ and
$w=2R(I+R^2)^{-1}$.

We now give the promised bidisk extension.  For scalar variables
$\zeta,\eta$ put
\begin{align*}
 q(\zeta,\eta)&=\zeta-\eta+i\sqrt2\,\zeta\eta,\\
 F(\zeta,\eta)&=\frac{(\zeta+\eta)^2}{2}
   +q(\zeta,\eta)\left(\frac{51i}{100}
                         +\frac{41(\eta-\zeta)}{400}\right).
\end{align*}
By \eqref{eq:CDrelations}, $q(C,D)=0$ and hence
$F(C,D)=w^2$.  We claim that
\begin{equation}\label{eq:Fbidisk}
             \norm{F}_{H^\infty(\mathbb D^2)}\le\kappa_0.
\end{equation}
Here is an exact elementary maximization.  On the distinguished boundary
write
\[
 \zeta=\chi e^{i\delta},\qquad
 \eta=\chi e^{-i\delta},\qquad |\chi|=1,
\]
and put $\tau=\sin\delta$, $x=\Re\chi$,
\[
 m=\frac{51}{50},\qquad n=\frac{41\sqrt2}{200},\qquad
 a(\tau)=2-\frac{51\sqrt2}{100}-\frac{159}{100}\tau^2.
\]
After removing a unit-modulus factor, direct expansion gives
\begin{equation}\label{eq:squaretorus}
 |F(\zeta,\eta)|=|-m\tau+a(\tau)\chi+n\tau\chi^2|,
\end{equation}
and hence
\begin{equation}\label{eq:squareconcave}
 |F(\zeta,\eta)|^2
 =a^2+2a\tau(n-m)x
   +\tau^2\bigl((m+n)^2-4mnx^2\bigr).
\end{equation}
For fixed $\tau$, the right side is a concave quadratic in
$x\in[-1,1]$.

Set $r=|\tau|$.  At $x=\pm1$, the larger modulus is
\begin{equation}\label{eq:squareedge}
                         |a(r)|+r(m-n).
\end{equation}
While $a(r)\ge0$, this concave quadratic in $r$ attains its maximum at
\[
 r_*=\frac{m-n}{2(159/100)}
     =\frac{17}{53}-\frac{41\sqrt2}{636},
\]
and that maximum is
\[
 2-\frac{51\sqrt2}{100}+\frac{(m-n)^2}{4(159/100)}
 =\frac{276889}{127200}-\frac{6103\sqrt2}{10600}=\kappa_0.
\]
While $a(r)\le0$, \eqref{eq:squareedge} is increasing and its endpoint
value $159/100-(2-51\sqrt2/100)+(m-n)$ is strictly smaller than
$\kappa_0$, since their difference is
\[
             \frac{199297-112032\sqrt2}{127200}>0.
\]

It remains to check the interior vertex of \eqref{eq:squareconcave}.  Such
a vertex belongs to $[-1,1]$ only if
\begin{equation}\label{eq:vertexcondition}
                         |a(r)|(m-n)\le4mnr.
\end{equation}
For $r<1/2$ this is impossible, since
\[
 a(1/2)(m-n)-2mn
 =\frac{147492-101353\sqrt2}{80000}>0.
\]
For $1/2\le r\le1$, the vertex value is
\[
 V(r^2)=(m+n)^2\left(r^2+\frac{a(r)^2}{4mn}\right).
\]
As a function of $r^2$, this is convex.  Exact endpoint evaluation gives
\begin{align*}
 \kappa_0^2-V(1/4)
   &=\frac{19(13707814025092-9187506155169\sqrt2)}
           {45109393920000}>0,\\
 \kappa_0^2-V(1)
   &=\frac{50692488845401-35843685592140\sqrt2}
           {11277348480000}>0.
\end{align*}
Thus every interior vertex is smaller than $\kappa_0^2$.  Equality in
\eqref{eq:Fbidisk} occurs in the edge calculation at $r=r_*$ and the
appropriate choice $x=-1$.  This proves the claimed exact norm.  The script
\path{tmp/research/right_angle_square_extension_exact.py} replays the
expansion and all comparisons over $\mathbb Q(\sqrt2)$.
In particular,
\[
 \sqrt2-\kappa_0
  =\frac{200436\sqrt2-276889}{127200}>0.
\]

Finally, Ando's dilation theorem for a commuting pair of contractions
\cite{Ando1963} and \eqref{eq:Fbidisk} give
\[
        \norm{w^2}=\norm{F(C,D)}
        \le\norm{F}_{H^\infty(\mathbb D^2)}\le\kappa_0,
\]
as required.
\end{proof}

\begin{proof}[Proof of \cref{thm:smallzeros}]
Keep the commuting contractions $C,D$ and the polynomial $F$ from the proof
of \cref{prop:andocentral}, and write
\[
                  W(\zeta,\eta)=\frac{\zeta+\eta}{\sqrt2}.
\]
For $\alpha,\beta\in\C$ define the rational bidisk function
\begin{equation}\label{eq:smallzeroextension}
 H_{\alpha,\beta}(\zeta,\eta)
 =\frac{F-(\alpha+\beta)W+\alpha\beta}
 {1-(\bar\alpha+\bar\beta)W+\bar\alpha\bar\beta F}.
\end{equation}
On the curve $q=0$ one has $F=W^2$, so the numerator and denominator in
\eqref{eq:smallzeroextension} factor respectively as
\[
 (W-\alpha)(W-\beta),\qquad
 (1-\bar\alpha W)(1-\bar\beta W).
\]
Since $q(C,D)=0$ and $W(C,D)=w$, it follows that
\begin{equation}\label{eq:smallzerotrace}
                  H_{\alpha,\beta}(C,D)
                   =b_\alpha(w)b_\beta(w).
\end{equation}

Put $\delta=1/150$, $r=|\alpha|$, and $s=|\beta|$.  On the closed bidisk,
$|W|\le\sqrt2$ and $|F|\le\kappa_0$.  Hence the denominator $D_0$ and
numerator $N_0$ in \eqref{eq:smallzeroextension} satisfy
\begin{align}
 |D_0|&\ge 1-(r+s)\sqrt2-rs\kappa_0
       \ge 1-2\delta\sqrt2-\delta^2\kappa_0,\label{eq:smalldenom}\\
 |N_0|&\le \kappa_0+(r+s)\sqrt2+rs
       \le \kappa_0+2\delta\sqrt2+\delta^2.\label{eq:smallnumer}
\end{align}
The first lower bound is strictly positive.  Indeed,
$0<\kappa_0<\sqrt2<3/2$ gives
\[
 1-2\delta\sqrt2-\delta^2\kappa_0
 >1-\frac1{50}-\frac1{15000}
 =\frac{14699}{15000}.
\]
Thus \eqref{eq:smallzeroextension} is analytic on a neighborhood of the
closed bidisk.

The exact reserve is also sufficient.  From the displayed value of
$\kappa_0$,
\[
 \sqrt2-\kappa_0
 =\frac{200436\sqrt2-276889}{127200}>\frac1{20};
\]
after clearing denominators, the last strict inequality follows from
\[
       2(200436)^2-(283249)^2=119184191>0.
\]
Consequently
\begin{align*}
 &\sqrt2(1-2\delta\sqrt2-\delta^2\kappa_0)
   -(\kappa_0+2\delta\sqrt2+\delta^2)\\
 &\qquad=(\sqrt2-\kappa_0)-(4+2\sqrt2)\delta
             -(1+\sqrt2\kappa_0)\delta^2\\
 &\qquad>\frac1{20}-\frac7{150}-\frac1{7500}
       =\frac2{625}>0.
\end{align*}
Equations \eqref{eq:smalldenom}--\eqref{eq:smallnumer} therefore give the
strict scalar norm bound displayed in \eqref{eq:smallzeros}.  Finally,
Ando's theorem applied to the commuting contractions $C,D$, together with
\eqref{eq:smallzerotrace}, gives the same bound for the operator rational
function: indeed, \eqref{eq:smallzeroextension} is analytic on a
neighborhood of the closed bidisk, so its Taylor polynomials converge there
uniformly and Ando's polynomial inequality passes to the limit.  The
endpoint choices $\alpha=1/150$, $\beta=i/150$ are allowed
because every estimate above holds on the closed parameter bidisk.
\end{proof}

\subsection{The exact remaining criterion for the centred slice}

The preceding certificates leave only a thin boundary annulus in the
$c$-parameter.  The
following reduction identifies its precise operator-theoretic content; in
particular, there is no additional formal branch in the two-moment region.

\begin{proposition}[Two-moment criterion]\label{prop:twomoment}
Let $T\in\B(\Hh)$ satisfy $W(T)\subset\cS_{\pi/4}$, put
\[
 w=(T^2-I)(T^2+I)^{-1},\qquad \mathsf X=w^2,
\]
and, for a unit vector $\xi$, put
\[
                 b=\norm{\mathsf X\xi}^2,
       \qquad q=\langle\mathsf X\xi,\xi\rangle.
\]
Then
\begin{equation}\label{eq:fullcdisk}
       \norm{h_c(w)}\le\sqrt2\qquad(c\in\mathbb D)
\end{equation}
holds if and only if
\begin{equation}\label{eq:twomoment}
 |q|^2\le(2-b)(2b-1)
 \quad\hbox{whenever}\quad b>1.
\end{equation}
Equivalently, the nonautomatic part of \eqref{eq:twomoment} is
\begin{equation}\label{eq:variance}
 \norm{\mathsf X\xi-q\xi}^2
       \ge 2\bigl(\norm{\mathsf X\xi}^2-1\bigr)^2,
 \qquad \norm{\mathsf X\xi}>1.
\end{equation}
\end{proposition}

\begin{proof}
By \cref{prop:andocentral},
$\norm{\mathsf X}\le\kappa_0<\sqrt2$, and hence
$0\le b\le\kappa_0^2<2$.  For $|c|<1$,
spectral mapping makes $I+c\mathsf X$ invertible.  Congruence by this
operator shows that \eqref{eq:fullcdisk} is equivalent to positivity, for
every $c\in\mathbb D$, of
\begin{equation}\label{eq:Hc}
 H(c)=(2-|c|^2)I+(2|c|^2-1)\mathsf X^*\mathsf X
                   +c\mathsf X+\bar c\mathsf X^*.
\end{equation}
Indeed, \eqref{eq:Hc} is the expansion of
\[
 2(I+c\mathsf X)^*(I+c\mathsf X)
       -(\mathsf X+\bar cI)^*(\mathsf X+\bar cI).
\]
Writing $c=re^{i\theta}$ and minimizing first over $\theta$ gives
\begin{equation}\label{eq:Hcompression}
 \min_\theta\langle H(re^{i\theta})\xi,\xi\rangle
       =2-b+(2b-1)r^2-2r|q|.
\end{equation}
Both endpoint values in $0\le r\le1$ are nonnegative: the value at zero is
$2-b$, while the value at one is
$b+1-2|q|\ge(\sqrt b-1)^2$.

If $b\le1/2$, the quadratic in \eqref{eq:Hcompression} is concave, so its
endpoint values suffice.  Suppose $1/2<b\le1$.  If its critical point
$r_*=|q|/(2b-1)$ is outside $(0,1)$, the endpoints again suffice.  If
$r_*\in(0,1)$, then
\[
 |q|^2\le(2b-1)^2\le(2-b)(2b-1),
\]
so the critical value is nonnegative as well.  Thus the range $b\le1$ is
automatic.

Finally suppose $b>1$.  Cauchy--Schwarz and
\[
                 (2b-1)^2-b=(4b-1)(b-1)>0
\]
give $|q|\le\sqrt b<2b-1$.  Hence the critical point belongs to $[0,1)$,
and its value is nonnegative exactly when
\[
                 2-b-\frac{|q|^2}{2b-1}\ge0,
\]
which is \eqref{eq:twomoment}.  This proves the equivalence.  The identity
\[
 b-(2-b)(2b-1)=2(b-1)^2
\]
and orthogonal projection onto $\mathbb C\xi$ give
\eqref{eq:variance}.
\end{proof}

\section{The sharp two-node admissible-kernel theorem}
\label{sec:twonode}

We record explicitly the two-node result used below.  For $a,b\in\mathbb D$
write
\[
 \delta(a,b)=\frac{|a-b|}{|1-\bar a b|}
\]
for the pseudohyperbolic distance, and put
\begin{equation}\label{eq:radialF}
 \mathcal F(t)=\frac{2\sqrt2\,t}{3+\sqrt{1-t^2}},
 \qquad 0\le t<1.
\end{equation}

\begin{lemma}[Radial scaling and right-angle geometry]
\label{lem:twonodegeometry}
For $a,b\in\mathbb D$,
\begin{equation}\label{eq:radialscaling}
 \delta(a/\sqrt2,b/\sqrt2)
 \le \mathcal F\bigl(\delta(a,b)\bigr).
\end{equation}
If $(c_j,d_j,w_j)\in\mathbb D^3$, $j=1,2$, satisfy
\begin{equation}\label{eq:rightanglecurve2}
 c_j-d_j+i\sqrt2c_jd_j=0,
 \qquad w_j=\frac{c_j+d_j}{\sqrt2},
\end{equation}
then
\begin{equation}\label{eq:rightanglegeometry}
 \max\{\delta(c_1,c_2),\delta(d_1,d_2)\}
 \ge \mathcal F\bigl(\delta(w_1,w_2)\bigr).
\end{equation}
\end{lemma}

\begin{proof}
We first prove \eqref{eq:radialscaling}.  More generally, fix $0<r<1$.
After a common rotation write $a=s\in[0,1)$ and
\[
 b=\frac{s+\xi}{1+s\xi},\qquad |\xi|\le t.
\]
Direct simplification gives
\[
 \delta(ra,rb)
 =\frac{r|\xi|(1-s^2)}
 {|1-r^2s^2+s(1-r^2)\xi|}.
\]
For fixed $s$ its maximum occurs at $\xi=-t$.  Differentiating in $s$
gives $ts^2-2s+t=0$, hence
$s=(1-\sqrt{1-t^2})/t$ when $t>0$.  The resulting maximum is
$2rs/(1+r^2s^2)$.  Taking $r=1/\sqrt2$ gives
\eqref{eq:radialscaling}; the case $t=0$ follows by continuity.

For \eqref{eq:rightanglegeometry}, undo the disk and half-plane Cayley
maps in \eqref{eq:rightanglecurve2} and write the two sector points as
$\beta_1=re^{i\theta}$ and $\beta_2=se^{i\phi}$, with
$|\theta|,|\phi|\le\pi/4$.  Set
\[
 U=\frac{r^2+s^2}{2rs},\qquad
 X=\cos(\phi-\theta),\qquad
 Y=|\sin(\phi+\theta)|.
\]
Then $U\ge1$ and $0\le Y\le X\le1$.  Cayley invariance and direct
calculation give, with
$m=\max\{\delta(c_1,c_2),\delta(d_1,d_2)\}$ and
$t=\delta(w_1,w_2)$,
\begin{equation}\label{eq:twonodeUXY}
 m^2=\frac{U-X}{U-Y},\qquad
 t^2=\frac{U^2-X^2}{U^2-Y^2}.
\end{equation}
If $\sigma=\sqrt{1-t^2}$, these identities imply
\[
 \frac{m^2}{t^2}=\frac{U+Y}{U+X},\qquad
 X^2=\sigma^2U^2+(1-\sigma^2)Y^2.
\]
Since
\[
 [\sigma U+(1+\sigma)Y]^2-X^2
 =2\sigma(1+\sigma)Y(U+Y)\ge0,
\]
we have $X\le\sigma U+(1+\sigma)Y$.  Therefore
\[
 \frac{m^2}{t^2}
 \ge\frac1{1+\sigma}
 \ge\frac8{(3+\sigma)^2},
\]
the last inequality being $(1-\sigma)^2\ge0$.  This is precisely
\eqref{eq:rightanglegeometry}.
\end{proof}

\begin{theorem}[Two-node right-angle completion]
\label{thm:twonode}
Let two triples satisfy \eqref{eq:rightanglecurve2}, let $h$ be Schur on
$\mathbb D$, and put $\lambda_j=h(w_j)$.  If a Hermitian $2\times2$ matrix
$L$ satisfies
\begin{equation}\label{eq:twonodeadmissible}
 (1-\bar c_i c_j)_{i,j=1}^2\circ L\succeq0,
 \qquad
 (1-\bar d_i d_j)_{i,j=1}^2\circ L\succeq0,
\end{equation}
then
\begin{equation}\label{eq:twonodetarget}
 (2-\bar\lambda_i\lambda_j)_{i,j=1}^2\circ L\succeq0.
\end{equation}
The same conclusion holds after independent disk automorphisms in the two
test coordinates, because their Pick kernels change only by invertible
diagonal congruences.
\end{theorem}

\begin{proof}
The assertion is immediate if a diagonal entry of $L$ vanishes, so assume
$L_{11}L_{22}>0$.  Positivity in \eqref{eq:twonodeadmissible} gives
\begin{equation}\label{eq:Ledgebound}
 \frac{|L_{12}|^2}{L_{11}L_{22}}
 \le 1-\max\{\delta(c_1,c_2)^2,\delta(d_1,d_2)^2\}.
\end{equation}
Schwarz--Pick, \cref{lem:twonodegeometry}, and monotonicity of
$\mathcal F$ give
\begin{align*}
 \delta(\lambda_1/\sqrt2,\lambda_2/\sqrt2)
 &\le\mathcal F(\delta(\lambda_1,\lambda_2))
 \le\mathcal F(\delta(w_1,w_2))\\
 &\le\max\{\delta(c_1,c_2),\delta(d_1,d_2)\}.
\end{align*}
Here monotonicity is immediate by differentiating \eqref{eq:radialF}; its
derivative has the positive numerator
$3+\sqrt{1-t^2}+t^2/\sqrt{1-t^2}$ on $0\le t<1$.
Using
\[
 \frac{|2-\bar\lambda_1\lambda_2|^2}
 {(2-|\lambda_1|^2)(2-|\lambda_2|^2)}
 =\frac1{1-\delta(\lambda_1/\sqrt2,
                         \lambda_2/\sqrt2)^2},
\]
inequality \eqref{eq:Ledgebound} is exactly the nonnegativity of the
determinant in \eqref{eq:twonodetarget}.  Its diagonal entries are positive,
so the target is positive semidefinite.
\end{proof}

\section{A complete symmetric three-node theorem}\label{sec:threenode}

We next record a finite but genuinely three-node instance of the sharp
right-angle multiplier problem.  Besides giving a new exact family, the
proof identifies the first algebraic obstruction beyond the two-node
theorem: a single phase on a rank-$(2,2)$ extreme face.

For $0<a<1$, put
\begin{equation}\label{eq:threecoords}
 \rho=(-a,0,a),\qquad
 z_i=\frac{2\rho_i}{1+\rho_i^2},\qquad
 u_i=\frac{\rho_i+i}{\sqrt2},\qquad
 v_i=\frac{\rho_i-i}{\sqrt2}.
\end{equation}
For a Hermitian $3\times3$ matrix $L$, all products below are entrywise.

\begin{theorem}[Symmetric three-node completion]\label{thm:symmetricthree}
If
\begin{equation}\label{eq:threeadmissible}
 P=(1-\bar u_i u_j)_{ij}\circ L\succeq0,\qquad
 Q=(1-\bar v_i v_j)_{ij}\circ L\succeq0,
\end{equation}
then
\begin{equation}\label{eq:threetarget}
                    (2-\bar z_i z_j)_{ij}\circ L\succeq0.
\end{equation}
Thus the identity disk multiplier has norm at most $\sqrt2$ on the complete
right-angle admissible-kernel cone over the nodes \eqref{eq:threecoords}.
\end{theorem}

We first give the two finite-dimensional ingredients.  For a Hermitian
correlation matrix $C$, set
\begin{equation}\label{eq:Delta}
 \Delta(C)=1-|c_{12}|^2-|c_{23}|^2-|c_{31}|^2
                   +2\Re(c_{12}c_{23}c_{31})=\det C.
\end{equation}

\begin{lemma}[Extreme-ray rank bound]\label{lem:rankbound}
Let $\mathcal R$ be an invertible real-linear Schur multiplier on
$M_n(\C)_{\rm sa}$, and let
\[
 \mathcal C=\{P=P^*:P\succeq0,\ \mathcal R(P)\succeq0\}.
\]
If $P$ generates an extreme ray of $\mathcal C$, then
\begin{equation}\label{eq:rankbound}
       \operatorname{rank}(P)^2+
       \operatorname{rank}(\mathcal R(P))^2\le n^2+1.
\end{equation}
\end{lemma}

\begin{proof}
Let $E_P,E_Q$ be the range projections of $P$ and
$Q=\mathcal R(P)$.  The real spans of their minimal positive-semidefinite
faces are
\[
 \mathcal S_P=\{H=H^*:H=E_PHE_P\},\qquad
 \mathcal S_Q=\{K=K^*:K=E_QKE_Q\},
\]
of dimensions $\operatorname{rank}(P)^2$ and
$\operatorname{rank}(Q)^2$.  Hence
\[
 \dim\bigl(\mathcal S_P\cap\mathcal R^{-1}(\mathcal S_Q)\bigr)
       \ge \operatorname{rank}(P)^2+
          \operatorname{rank}(Q)^2-n^2.
\]
If this intersection contained a Hermitian $H$ independent of $P$, then
$P\pm\varepsilon H$ and
$Q\pm\varepsilon\mathcal R(H)$ would be positive for sufficiently small
$\varepsilon>0$: on their ranges $P,Q$ are positive definite, and the
perturbations vanish on their kernels.  This would split the ray of $P$.
The intersection is therefore one-dimensional, which proves
\eqref{eq:rankbound}.
\end{proof}

\begin{lemma}[Exact phase elimination]\label{lem:phaseelimination}
Let $p$ be a rank-at-most-two correlation matrix.  After diagonal unitary
conjugation, write
\begin{equation}\label{eq:phaseparam}
 p_{12}=x,\quad p_{23}=y,\quad
 p_{31}=xy+\sqrt{(1-x^2)(1-y^2)}\,\zeta,\qquad |\zeta|=1,
\end{equation}
where $0\le x,y\le1$.  Put $X=x^2$, $Y=y^2$ and
$C_0=2xy\sqrt{(1-X)(1-Y)}$.  For a normalized Hermitian multiplier $m$,
define
\begin{align}
 M_m&=m_{12}m_{23}m_{31},&
 d_m&=M_m-|m_{31}|^2,\label{eq:md}\\
 A_m&=1-|m_{12}|^2X-|m_{23}|^2Y\notag\\
 &\quad-|m_{31}|^2\{XY+(1-X)(1-Y)\}
             +2XY\Re M_m.\label{eq:Am}
\end{align}
Then
\begin{equation}\label{eq:phaseaffine}
                 \Delta(m\circ p)=A_m+C_0\Re(d_m\zeta).
\end{equation}
Suppose $C_0>0$, $d_r\ne0$, and $\Delta(r\circ p)=0$.  With
\begin{align}\label{eq:FHS}
 \gamma&=\frac{d_t}{d_r},&
 F&=A_t-A_r\Re\gamma,\notag\\
 H&=C_0^2|d_r|^2-A_r^2,&
 S&=F^2-H(\Im\gamma)^2.
\end{align}
both admissible phases satisfy $\Delta(t\circ p)\ge0$ if and only if
\begin{equation}\label{eq:phasecriterion}
                         F\ge0,\qquad S\ge0.
\end{equation}
\end{lemma}

\begin{proof}
Expanding \eqref{eq:Delta} after \eqref{eq:phaseparam} gives
\eqref{eq:phaseaffine}.  Put $\eta=d_r\zeta$.  The equation
$\Delta(r\circ p)=0$ fixes
\[
 \Re\eta=-A_r/C_0,\qquad |\eta|=|d_r|,
\]
and hence leaves two conjugate choices for its imaginary part.  Since
$d_t\zeta=\gamma\eta$, their smaller target determinant is exactly
\begin{equation}\label{eq:phasemin}
                       F-\sqrt H\,|\Im\gamma|.
\end{equation}
Feasibility gives $H\ge0$.  The expression in \eqref{eq:phasemin} is
nonnegative precisely when $F\ge0$ and its square is nonnegative, which is
\eqref{eq:phasecriterion}.  The sign condition on $F$ prevents an
extraneous conclusion from squaring.
\end{proof}

\begin{proof}[Proof of \cref{thm:symmetricthree}]
All entries of $k_u=(1-\bar u_i u_j)$ are nonzero.  The change of variable
$P=k_u\circ L$ identifies \eqref{eq:threeadmissible} with
\[
 \mathcal C=\{P\succeq0:\mathcal R(P)=r\circ P\succeq0\},
 \qquad r_{ij}=\frac{1-\bar v_i v_j}{1-\bar u_i u_j}.
\]
Intersecting this cone with
$\operatorname{tr}P+\operatorname{tr}\mathcal R(P)=1$ gives a compact base.
For a fixed vector, the
compression of the target in \eqref{eq:threetarget} is linear in $P$.
Thus a negative minimum, if one existed, would occur on an extreme ray.
By \cref{lem:rankbound}, its rank pattern is one of
\begin{equation}\label{eq:ranklist}
                (1,1),(1,2),(2,1),(2,2),(1,3),(3,1).
\end{equation}

We first treat the rank-$(2,2)$ pattern.  Positive diagonal congruence
normalizes $P,Q$ and the target to correlation matrices
$p,r\circ p,t\circ p$.  Their one- and two-index target minors are positive
by \cref{thm:twonode}.  It remains only to prove
$\Delta(t\circ p)>0$.

If a diagonal entry of $P$ or $Q$ vanishes, positivity makes the
corresponding row and column vanish; because the localization factors are
nonzero, the same is true of $L$ and the target.  This reduces directly to
\cref{thm:twonode}.  We may therefore assume all diagonal entries are
positive when making the correlation normalization above.

Put $q=a^2$.  Exact substitution of \eqref{eq:threecoords} into
\cref{lem:phaseelimination} gives
\begin{align}
 |t_{12}|^2=|t_{23}|^2
   &=\frac{1-q^2}{1+q^2},\label{eq:tshort}\\
 |t_{31}|^2
   &=\frac{(1-q)^2(q^2+4q+1)^2}
          {(q^2+1)^2(q^2+6q+1)},\label{eq:tlong}\\
 \Re d_r&=-\frac{32q^3}{(q+1)^2(q^2+6q+1)},\qquad
 |d_r|^2=\frac{64q^3}{(q+1)^2(q^2+6q+1)},\label{eq:drfamily}\\
 \Re\gamma&=\frac{|t_{31}|^2}{2},\qquad
 (\Im\gamma)^2=
 \frac{q(1-q)^4(q^2+4q-1)^2(q^2+4q+1)^2}
 {16(q^2+1)^4(q^2+6q+1)^2}.\label{eq:gammafamily}
\end{align}
The two quantities in \eqref{eq:phasecriterion} now have exact polynomial
certificates.  First, if
\[
 D_F=(q+1)^2(q^2+1)^2(q^2+6q+1)^2,
\]
the bilinear polynomial $D_FF$ has tensor Bernstein control matrix
\begin{equation}\label{eq:Fcontrols}
 \begin{pmatrix}
 32q^3(q+1)^2(q^2+6q+1)&
 2q^2(q+1)^2(q^2+1)(q^2+6q+1)^2\\
 2q^2(q+1)^2(q^2+1)(q^2+6q+1)^2&16q^3J(q)
 \end{pmatrix},
\end{equation}
where
\[
 J(q)=3q^6+22q^5+35q^4-12q^3+9q^2+6q+1.
\]
Its degree-six Bernstein controls on $[0,1]$ are
\[
 1,2,\frac{18}{5},\frac{26}{5},\frac{128}{15},\frac{64}{3},64.
\]
Every entry in \eqref{eq:Fcontrols} is therefore positive, so $F>0$.

Next write
\begin{equation}\label{eq:Snum}
 S=\frac{\mathcal N_q(X,Y)}
 {(q+1)^2(q^2+1)^4(q^2+6q+1)^2}.
\end{equation}
The bidegree-$(2,2)$ tensor Bernstein controls of $\mathcal N_q$ are the
symmetric matrix determined by
\begin{align*}
 b_{00}&=1024q^6(q+1)^2,\\
 b_{01}&=64q^5(q+1)^2(q^2+1)(q^2+6q+1),\\
 b_{02}&=4q^4(q+1)^2(q^2+1)^2(q^2+6q+1)^2,\\
 b_{11}&=-2q^4P_{10}(q),\\
 b_{12}&=32q^5(q^2+1)J(q),\\
 b_{22}&=256q^6P_1(q)P_2(q),
\end{align*}
where
\begin{align*}
 P_{10}(q)={}&q^{10}-2q^9-63q^8-224q^7-514q^6-732q^5\\
 &{}-66q^4-288q^3-127q^2-34q+1,\\
 P_1(q)&=q^3+3q^2-q+1,\\
 P_2(q)&=q^4+4q^3-2q^2+1.
\end{align*}
The Bernstein controls of $P_1,P_2$ on $[0,1]$ are respectively
\[
 (1,2/3,4/3,4),\qquad (1,1,2/3,1,4).
\]
For $q\ge1/33$,
$P_{10}(q)<q^{10}+1-34q\le1-33q\le0$; hence every $b_{ij}$ is positive.

For the remaining small parameters put $X=m+d$, $Y=m-d$.  Direct
factorization gives
\begin{equation}\label{eq:smallqfactor}
 \frac{\mathcal N_q(m+d,m-d)}{4q^4}=D_0(m,q)+d^2D_2(m,q),
\end{equation}
where
\begin{align*}
 D_0={}&32q[mq-m+q+1]\\
 &\quad\cdot[mq^3+3mq^2-3mq-m+2q+2]\\
 &\quad\cdot[mq^4+4mq^3-2mq^2-4mq+m+4q],\\
 D_2={}&4(q-1)^2(q^2+4q-1)(q^2+4q+1)\\
 &\quad\cdot[-8mq^3+8mq+q^4-8q-1].
\end{align*}
For $0<q\le1/5$ and $0\le m\le1$, the three factors in $D_0$ are
positive.  The two sign-changing factors in $D_2$ are both negative, so
$D_2\ge0$.  Thus $S>0$ for $q\le1/5$.  Since
$(0,1)=(0,1/5]\cup[1/33,1)$, \eqref{eq:phasecriterion} proves the desired
rank-$(2,2)$ determinant for all $0<a<1$.  If $C_0=0$, the second rank
condition and \eqref{eq:drfamily} force $X=0$ or $Y=0$; the remaining
determinant is affine and its endpoint gaps are
\[
 1-|t_{31}|^2=\frac{32q^3}{(q^2+1)^2(q^2+6q+1)}>0,\qquad
 1-|t_{12}|^2=\frac{2q^2}{1+q^2}>0.
\]

It remains to handle the five patterns in \eqref{eq:ranklist} with a
rank-one side.  Apply the disk automorphisms
\begin{equation}\label{eq:cdcoords}
 c_i=\frac{\sqrt2\rho_i}{1+i\rho_i},\qquad
 d_i=\frac{\sqrt2\rho_i}{1-i\rho_i}.
\end{equation}
They are automorphisms of the $u$- and $v$-disks, respectively, so their
Pick kernels differ from those in \eqref{eq:threeadmissible} by invertible
diagonal congruences.  Moreover
\begin{equation}\label{eq:curveidentity}
 c_i-d_i+i\sqrt2c_id_i=0,\qquad
 z_i=\frac{c_i+d_i}{\sqrt2},\qquad
 1-\bar z_i z_j=(1-\bar c_i c_j)(1-\bar d_i d_j).
\end{equation}

Suppose the $c$-localized matrix has rank one and full support.  Diagonal
congruence reduces positivity of the $d$-localized matrix to
\[
 \left[\frac{1-\bar d_i d_j}{1-\bar c_i c_j}\right]\succeq0.
\]
The scalar Pick theorem supplies a Schur function $\varphi$ with
$\varphi(c_i)=d_i$.  On the finite reproducing-kernel space with Gram
matrix $[(1-\bar c_i c_j)^{-1}]$, both $c$ and $d=\varphi(c)$ are
contractive multipliers.  Since $z=(c+d)/\sqrt2$,
\[
                          \|M_z\|\le\sqrt2.
\]
The multiplier Pick criterion is exactly \eqref{eq:threetarget}.  If the
rank-one factor has a zero coordinate, the corresponding target row and
column vanish and the assertion reduces to at most two nodes, where
\cref{thm:twonode} applies.  Interchanging
$c,d$ treats a rank-one second side.  Thus every extreme ray in
\eqref{eq:ranklist} satisfies the target.  Convexity of the compact base and
homogeneity now prove \eqref{eq:threetarget} for every admissible $L$.
\end{proof}

The phase calculation also gives a genuinely nonsymmetric patch.  Its
proof is computer assisted only in the finite sense that a displayed family
of rational Bernstein controls is reconstructed and checked in exact
arithmetic.

\begin{theorem}[Asymmetric three-node patch]\label{thm:asymmetricthree}
Let
\[
\rho=(-a,0,b),\qquad \frac16\le a,b\le\frac56,
\]
and put
$z_i=2\rho_i/(1+\rho_i^2)$,
$u_i=(\rho_i+i)/\sqrt2$, and $v_i=(\rho_i-i)/\sqrt2$.  Every
Hermitian $L$ satisfying \eqref{eq:threeadmissible} obeys
\eqref{eq:threetarget}.  Hence the identity multiplier has norm at most
$\sqrt2$ on the complete right-angle admissible-kernel cone for this
two-parameter family.
\end{theorem}

\begin{proof}
The compact-base and extreme-ray arguments above are unchanged.  Every
extreme ray with a rank-one side is handled by
\eqref{eq:cdcoords}--\eqref{eq:curveidentity} and scalar Pick interpolation,
and a zero diagonal reduces to at most two nodes.  It remains to treat the
rank-$(2,2)$ pattern.

Put
\begin{align*}
 A&=1+a^4,&B&=1+b^4,\\
 G&=a^2b^2+a^2+4ab+b^2+1,&
 K&=a^2b^2+(a+b)^2+1.
\end{align*}
Exact normalization of the three target edges gives
\begin{equation}\label{eq:asymedges}
 \alpha=\frac{1-a^4}{A},\qquad
 \beta=\frac{1-b^4}{B},\qquad
 \gamma_e=\frac{(1-a^2)(1-b^2)K^2}{ABG}.
\end{equation}
For the phase quantities of \cref{lem:phaseelimination}, direct kernel
substitution gives
\begin{align}\label{eq:asymphase}
 \Re d_r&=-\frac{8a^2b^2(a+b)^2}{(1+a^2)(1+b^2)G},\\
 |d_r|^2&=\frac{16a^2b^2(a+b)^2}{(1+a^2)(1+b^2)G},\notag\\
 \Re\gamma&=\frac{\gamma_e}{2},\qquad
 \Im\gamma=\frac{ab(1-a^2)(1-b^2)EK}
 {2(a+b)ABG},\notag
\end{align}
where $E=a^2b^2+(a+b)^2-1$.

Regard the numerators of $F$ and $S$ in \eqref{eq:FHS} as polynomials in
$X,Y\in[0,1]$.  Their tensor Bernstein control arrays have sizes $2$ by
$2$ and $3$ by $3$.  Exact factorization reduces all thirteen controls to
positive elementary factors and four parameter polynomials $U,V,W,Z$; the
only sign requirements are
\begin{equation}\label{eq:asymUV}
                         U>0,\qquad V>0,\qquad W<0,\qquad Z>0.
\end{equation}
The first one has the global sum-of-squares identity
\begin{equation}\label{eq:asymU}
 U=4a^2b^2+(a-b)^2\{(a+b)^2+4ab+(1-ab)^2\}>0.
\end{equation}
For the remaining three, set
$a=1/6+2\sigma/3$, $b=1/6+2\tau/3$.  Their exact tensor Bernstein
expansions on $(\sigma,\tau)\in[0,1]^2$ have respectively $49$, $169$,
and $81$ rational controls.  Every control of $V,-W,Z$ is positive; their
least values are
\begin{equation}\label{eq:asymcontrols}
 \frac{851394073}{725594112},\qquad
 \frac{177918726900023}{4738381338321616896},\qquad
 \frac{76259197261}{705277476864}.
\end{equation}
The exact verifier cited in the reproducibility statement reconstructs
\eqref{eq:asymedges}--\eqref{eq:asymphase}, divides the thirteen controls
by their displayed positive factors with zero remainder, and computes all
$299$ rational controls from the monomial-to-Bernstein formula.  Thus
\eqref{eq:asymUV} holds and the first Bernstein layer gives $F>0$, $S>0$
for every $X,Y\in[0,1]$.

If $C_0>0$, \cref{lem:phaseelimination} now gives
$\Delta(t\circ p)>0$.  If $C_0=0$, the second rank condition forces
$X=0$ or $Y=0$, and the target determinant is respectively
\[
 (1-\gamma_e)(1-Y)+(1-\beta)Y,
 \qquad
 (1-\gamma_e)(1-X)+(1-\alpha)X.
\]
These are positive by \eqref{eq:asymedges} and
$ABG-(1-a^2)(1-b^2)K^2=2(a+b)^2U$.  The target minors of order two are
positive for the same reason.  Hence every rank-$(2,2)$ extreme ray, and
therefore the full cone, satisfies \eqref{eq:threetarget}.
\end{proof}

\begin{remark}\label{rem:threescope}
\Cref{thm:symmetricthree,thm:asymmetricthree} control whole
admissible-kernel cones, not only rank-$(2,2)$ faces.  Their scope is
nevertheless precise: the multiplier is the identity, and the nonsymmetric
theorem covers only the displayed parameter box.  The general three-node
problem still contains the rank-$(2,2)$ phase inequality outside this box
and a degree-two Blaschke determinant for arbitrary Schur data.  The
arbitrary-node right-angle constant therefore remains open.
\end{remark}

\section{A complex-zero quadratic singular family}
\label{sec:realzero}

The rank-one localized faces reduce to scalar Pick matrices on three nodes,
and their extreme scalar data are quadratic Blaschke products.  We now prove
the sharp target on a full two-real-dimensional disk of such products.  Put
\begin{equation}\label{eq:realzeroparams}
 \theta\in\mathbb R,\qquad
 \eta=e^{i\theta},\qquad p=\cos\theta,\qquad q=\sin\theta,
 \qquad \frac1{\sqrt2}<q\le1.
\end{equation}
Let $c_0=d_0=w_0=0$, let $c_1,c_2$ be the roots of
\begin{equation}\label{eq:realzeroroots}
 -i\sqrt2\eta c^2+\eta c-1=0,
\end{equation}
and set
\begin{equation}\label{eq:realzerocurve}
 d_j=\eta c_j^2=\frac{c_j}{1-i\sqrt2c_j},\qquad
 w_j=\frac{c_j+d_j}{\sqrt2}.
\end{equation}
For $a\in\mathbb D$, define
\begin{equation}\label{eq:realzeroB}
 B_a(w)=w\frac{w-a}{1-\bar a w}.
\end{equation}

\begin{theorem}[Complex-zero quadratic singular family]
\label{thm:realzero}
For the nodes in \eqref{eq:realzeroparams}--\eqref{eq:realzerocurve},
\begin{equation}\label{eq:realzeroR}
 R_{jk}=\frac{1-\bar d_jd_k}{1-\bar c_jc_k}
\end{equation}
is positive semidefinite of rank two.  For every $a\in\mathbb D$, the actual
target
\begin{equation}\label{eq:realzeroT}
 T_{jk}=\frac{2-\overline{B_a(w_j)}B_a(w_k)}
              {1-\bar c_jc_k},\qquad 0\le j,k\le2,
\end{equation}
is positive definite.  The same conclusion holds for a unimodular multiple
of $B_a$.
\end{theorem}

\begin{proof}
Vieta's formulas and a direct modulus calculation give, with
$\rho=\sqrt{33-8\sqrt2q}$,
\begin{equation}\label{eq:realzeroVieta}
 c_1+c_2=-\frac{i}{\sqrt2},\qquad
 c_1c_2=-\frac{i}{\sqrt2\eta},\qquad
 |c_1-c_2|^2=\frac\rho2,
\end{equation}
and $|c_1|^2+|c_2|^2=(1+\rho)/4$, $|c_1c_2|^2=1/2$.
Thus $1/2<|c_j|^2<1$.  The change
\[
 c=\frac{\sqrt2\xi}{1+i\xi},\qquad
 d=\frac{\sqrt2\xi}{1-i\xi},\qquad
 w=\frac{2\xi}{1+\xi^2}
\]
then shows that $c_j,d_j,w_j\in\mathbb D$.  Moreover
\[
 R_{jk}=\frac{1-\bar c_j^2c_k^2}{1-\bar c_jc_k}
       =1+\bar c_jc_k,
\]
so $R$ is the rank-two Gram matrix of the distinct vectors $(1,c_j)$.

Because $B_a(0)=0$, take the Schur complement of $T$ at $T_{00}=2$.
For $j=1,2$, put
\[
 h_j=\frac{B_a(w_j)}{c_j},\qquad
 M_{jk}=\frac{2-\bar h_jh_k}{1-\bar c_jc_k}.
\]
Then $T_{jk}-2=\bar c_jc_kM_{jk}$ and
$\det T=2|c_1c_2|^2\det M=\det M$.  Also
\[
 |h_j|=\frac{|w_j|}{|c_j|}
 \left|\frac{w_j-a}{1-\bar a w_j}\right|<\sqrt2.
\]
Hence the diagonal entries of $M$ are positive, and it is enough to prove
$\det M>0$.

On the two roots of \eqref{eq:realzeroroots},
\begin{equation}\label{eq:realzeroaffine}
 c^2=\frac{i}{\sqrt2\eta}-\frac{i}{\sqrt2}c,\qquad
 w=\frac i2+\left(\frac1{\sqrt2}-\frac{i\eta}{2}\right)c,\qquad
 \frac wc=\frac1{\sqrt2}+\frac\eta{\sqrt2}c.
\end{equation}
Thus $h=(w/c)(w-a)/(1-\bar a w)$ has a unique affine representative
$A+Lc$.  Its Cramer denominator is
\begin{equation}\label{eq:realzeroDelta}
 \Delta=(1-\bar a w_1)(1-\bar a w_2)\ne0.
\end{equation}
Put $y_j=h_j/\sqrt2$ and
\[
 Z=|c_1-c_2|^2,\quad C=|1-\bar c_1c_2|^2,\quad
 U=|1-\bar y_1y_2|^2,\quad \ell=|L|^2/2.
\]
The two-point determinant identity is
\begin{equation}\label{eq:realzeroPickdet}
 \det P_y=\frac{Z\Xi}{C(C-Z)},\qquad
 \Xi=U-\ell C,\qquad M=2P_y.
\end{equation}

Write $a=u+iv$ and $x=|a|^2=u^2+v^2<1$.  Exact reduction of
\eqref{eq:realzeroPickdet} by \eqref{eq:realzeroVieta} and
\eqref{eq:realzeroaffine} gives the root-free formula
\begin{equation}\label{eq:realzeroXi}
 \Xi=\frac{N}{16H},\qquad H=8|\Delta|^2>0,\qquad N=E+puK,
\end{equation}
where
\begin{align}
E={}&71-8q^2-28\sqrt2q+x(24q^2-4\sqrt2q-41)\notag\\
 &+v(-8\sqrt2q^3-44q^2+44\sqrt2q-30)
 +vx(48q^2+4\sqrt2q-40)\notag\\
 &+v^2(-32q^2+8\sqrt2q+80),                         \label{eq:realzeroE}\\
K={}&4\{2\sqrt2q^2-5q+12\sqrt2+8qv-10\sqrt2v
                   -x(4q+7\sqrt2)\},                \label{eq:realzeroK}\\
H={}&8-5x+3x^2-8v+16v^2-6vx+4\sqrt2qv+8\sqrt2qv^2\notag\\
 &+2\sqrt2qx^2-6\sqrt2qx+4q^2vx
   +4\sqrt2pu(1-2v)-4pqux.                          \label{eq:realzeroH}
\end{align}
Furthermore
\begin{equation}\label{eq:realzeroCZ}
 C=\frac{5+\rho}{4},\qquad C-Z=\frac{5-\rho}{4},
 \qquad C(C-Z)=\frac{\sqrt2q-1}{2},
\end{equation}
and therefore
\begin{equation}\label{eq:realzerodetT}
 \det T=\frac{\sqrt{33-8\sqrt2q}\,N}
 {4H(\sqrt2q-1)}.
\end{equation}

We prove $N>0$.  The coefficient $E_x$ is affine in $v$, and
\[
 E_x(q,1)=72q^2-81\le-9,\qquad
 E_x(q,-1)=-24q^2-8\sqrt2q-1<0.
\]
Thus $E(q,x,v)\ge E(q,1,v)$.  The latter is a quadratic
$A(q)v^2+B(q)v+C(q)$ with $A(q)=8(10+\sqrt2q-4q^2)>0$.
The seven Bernstein controls of $4AC-B^2$ on
$q\in[1/\sqrt2,1]$ are
\[
\begin{gathered}
1152,\quad32(53-17\sqrt2),\quad
\frac{32(1179-596\sqrt2)}{15},\quad
\frac{8(2067-1229\sqrt2)}5,\\
\frac{4(12895-8172\sqrt2)}{15},\quad
\frac{4(1869-1210\sqrt2)}3,\quad4(-193+152\sqrt2).
\end{gathered}
\]
They are strictly positive by rational squaring, whence
\begin{equation}\label{eq:complexzeroEpositive}
 E>0.
\end{equation}

Since $p^2=1-q^2$ and $u^2=x-v^2$, set
\begin{equation}\label{eq:complexzeroF}
 F(q,x,v)=E^2-(1-q^2)(x-v^2)K^2.
\end{equation}
This is cubic in $x$.  Under
$x=v^2+(1-v^2)s$, $0\le s<1$, its exact cubic Bernstein expansion is
\begin{equation}\label{eq:complexzeroBernstein}
 F=B_0(1-s)^3+3B_1s(1-s)^2+3B_2s^2(1-s)+B_3s^3,
\end{equation}
where
\begin{align*}
B_0&=F(q,v^2,v)=E(q,v^2,v)^2,\\
B_1&=B_0+(1-v^2)F_x(q,v^2,v)/3,\\
B_2&=F(q,1,v)-(1-v^2)F_x(q,1,v)/3,\\
B_3&=F(q,1,v).
\end{align*}
We certify $B_1>0$ and $B_2,B_3\ge0$ on
$[1/\sqrt2,1]\times[-1,1]$.

Put $t=q-1/\sqrt2$ and $h=1-1/\sqrt2$.  Exact tensor Bernstein
conversion over $\mathbb Q(\sqrt2)$ gives positive controls for $B_1$ on
\[
 [0,h]\times[-1,0],\quad[0,h]\times[0,1/2],\quad
 [0,h]\times[1/2,1].                                 \tag{5.1}
\]
There are 147 controls, all strictly positive.  Next factor
\begin{equation}\label{eq:complexzeroPG}
 B_3=128(q-5\sqrt2/4)^2P(q,v),\qquad
 B_2=(5\sqrt2/4-q)G(q,v).
\end{equation}
For $P$, use $[0,h]\times[-1,0]$; each of the six $v$-intervals
\[
 [0,1/32],[1/32,1/16],[1/16,1/8],[1/8,1/4],
 [1/4,1/2],[1/2,1]
\]
over $t\in[1/5,h]$; and the last five of these intervals over
$t\in[0,1/5]$.  The resulting 12 boxes have 294 positive and six zero
controls.  The omitted local rectangle is positive because, for
$0\le t\le1/5$, $0\le v\le1/32$,
\begin{align}
P={}&(t^2-3v/2)^2+(4-26v-11\sqrt2t)v^2+40v^4\notag\\
 &+t(22\sqrt2v^3+16\sqrt2v^4)
 +t^2(6v^2+24v^3)+t^3(2\sqrt2v+4\sqrt2v^2),          \label{eq:complexzeroPlocal}
\end{align}
and $4-26v-11\sqrt2t\ge51/16-11\sqrt2/5>0$; the last
comparison is equivalent to $255^2>2\cdot176^2$.

For $G$, use $[0,h]\times[-1,0]$, $[0,1/64]\times[1/64,1]$, and
the following pairs of $v$-intervals:
\[
\begin{array}{c|c}
t\text{-interval}&v\text{-intervals}\\ \hline
{[1/64,1/32]}&[0,1/16],\ [1/16,1]\\
{[1/32,1/16]}&[0,1/8],\ [1/8,1]\\
{[1/16,1/8],\ [1/8,1/5],\ [1/5,h]}&[0,1/4],\ [1/4,1]\\
\end{array}
\]
These 12 boxes have 429 positive and three zero controls.  On the omitted
square $0\le t,v\le\delta:=1/64$, write
\[
 G=192\sqrt2t^2-\frac{512}{3}tv+1752\sqrt2v^2+R.
\]
The quadratic part is at least $2752(t^2+v^2)/15$.  Every monomial of $R$
has total degree at least three and is divisible by $t^2$ or $v^2$; using
$\sqrt2<3/2$, the sums of coefficient majorants in degrees $3$ through $8$
are
\[
 \frac{28216}{3},\ 10528,\ 4488,\ \frac{19120}{3},
 \frac{11072}{3},\ \frac{2816}{3}.
\]
Consequently
\begin{equation}\label{eq:complexzeroGlocal}
 |R|\le\frac{40143488841}{268435456}(t^2+v^2)
       <\frac{2752}{15}(t^2+v^2),
\end{equation}
so $G\ge0$ there as well.

Thus the final rectangular certificate contains $879$ exact controls:
$870$ are positive and nine are zero.  (The preliminary count $574$
corresponds to a coarser box list whose Bernstein controls are not all
nonnegative, and is not the proof certificate.)  No sign decision uses
floating-point arithmetic.

If $s=0$, then $F=B_0=E^2>0$; if $0<s<1$, the $B_1$ term in
\eqref{eq:complexzeroBernstein} is strictly positive and the other terms are
nonnegative.  Hence $F>0$, so \eqref{eq:complexzeroEpositive} gives
$|puK|<E$ and $N=E+puK>0$.  Formula \eqref{eq:realzerodetT} now proves
$\det M>0$ and completes the Schur-complement proof.
\end{proof}

\begin{remark}\label{rem:realzeroscope}
The lower endpoint $q=1/\sqrt2$ is excluded: one node reaches the unit
circle and $C(C-Z)=0$.  The endpoint $q=1$ is included and has $p=0$.
The strict condition $|a|<1$ ensures both $\Delta\ne0$ and $s<1$ above;
$|a|=1$ is used only as a closure face in the polynomial certificate.
The proof treats the complex phase directly rather than reducing to the
real-zero theorem.  The theorem nevertheless concerns only this
singular arc and products with one zero fixed at the origin.  It does not
settle arbitrary singular nodes, arbitrary quadratic inner functions, or
the unrestricted fixed-lens constant.
\end{remark}
\section{The unrestricted angle-dependent envelope}\label{sec:bounds}

For context, we now place \cref{thm:main} inside the best bounds that can be
verified from the literature cited below.  Define
\begin{align}
 L(\alpha)&:=\frac{\pi\sin\alpha}{2\alpha},\label{eq:L}\\
 K_{25}(\alpha)&:=1-\frac\alpha\pi+
 \sqrt{2-\frac{4\alpha}{\pi}+\frac{\alpha^2}{\pi^2}},\label{eq:K25}\\
 \mu(\alpha)&:=\frac{\sin(2\alpha)}{\pi}
 \int_0^\infty
 \frac{dy}{y^2\cos\alpha-2y\cos(2\alpha)+\cos\alpha},\label{eq:mu}\\
 K_{\mathrm{BC}}(\alpha)&:=2\frac{\pi-\alpha}{\pi}+\mu(\alpha).
\end{align}
The integral in \eqref{eq:mu} is positive and finite for
$0<\alpha<\pi/2$; its endpoint values are understood by limits.

\begin{proposition}[Verified envelope]\label{prop:envelope}
For $0<\alpha\le\pi/2$,
\begin{equation}\label{eq:envelope}
 L(\alpha)\le C(\alpha)\le
 \min\left\{
    K_{25}(\alpha),\ K_{\mathrm{BC}}(\alpha),\
    \frac{\pi-\alpha}{\alpha}
 \right\}.
\end{equation}
In particular,
\begin{equation}\label{eq:sub2}
                    C(\alpha)<2
        \qquad\text{whenever}\qquad \alpha>\frac\pi6.
\end{equation}
\end{proposition}

\begin{proof}
The lower bound is \cref{cor:lenscert}.  The term
$(\pi-\alpha)/\alpha$ is the classical sector bound recorded in
\cite{MR2223270}; the conic-domain argument of
\cite{Beckermann2007} yields $K_{\mathrm{BC}}$.  Formula
\eqref{eq:K25} is the aperture bound proved in the preprint
\cite{crouzeix2025numericalrangesspectralsets}.  By \cref{prop:equiv}, all three sector bounds
apply to $C(\alpha)$.  Finally, direct algebra gives
$K_{25}(\alpha)<2$ exactly when $\alpha>\pi/6$.
\end{proof}

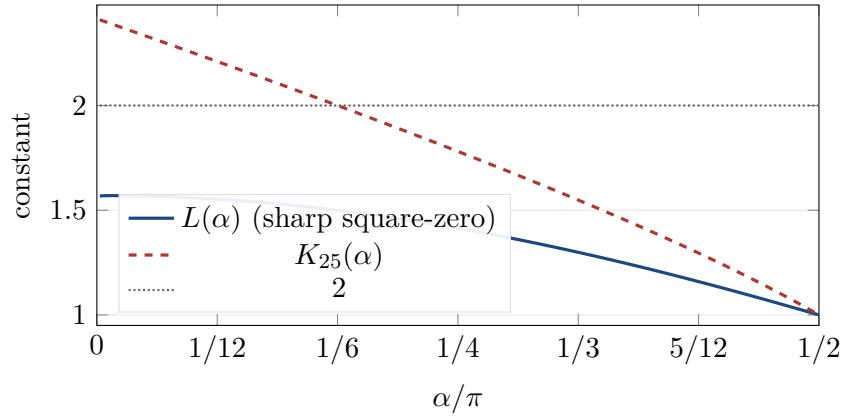
\begin{figure}[t]
\centering
\begin{tikzpicture}
\begin{axis}[
  width=.88\textwidth,
  height=.46\textwidth,
  xmin=0,xmax=.5,
  ymin=.95,ymax=2.48,
  xlabel={$\alpha/\pi$},
  ylabel={constant},
  xtick={0,.083333,.166667,.25,.333333,.416667,.5},
  xticklabels={$0$,$1/12$,$1/6$,$1/4$,$1/3$,$5/12$,$1/2$},
  ymajorgrids=true,
  grid style={gray!18},
  legend style={
    draw=gray!25,
    fill=white,
    fill opacity=.92,
    text opacity=1,
    at={(.03,.04)},
    anchor=south west
  },
  domain=.002:.5,
  samples=180
]
\addplot[deepblue,very thick] {sin(deg(pi*x))/(2*x)};
\addlegendentry{$L(\alpha)$ (sharp square-zero)}
\addplot[warmred,very thick,dashed] {1-x+sqrt(2-4*x+x^2)};
\addlegendentry{$K_{25}(\alpha)$}
\addplot[black!60,densely dotted,thick] {2};
\addlegendentry{$2$}
\end{axis}
\end{tikzpicture}
\caption{The exact square-zero constant and the 2025 aperture upper bound.
The vertical gap is the part not resolved by first-order nilpotent models.
The line $K_{25}=2$ is crossed at $\alpha/\pi=1/6$.}
\label{fig:bounds}
\end{figure}

At representative angles the two explicit curves in \cref{fig:bounds} are:
\begin{center}
\begin{tabular}{@{}ccccc@{}}
\toprule
$\alpha$ & $\pi/12$ & $\pi/6$ & $\pi/4$ & $\pi/3$ \\
\midrule
$L(\alpha)$ & $1.5529$ & $1.5000$ & $1.4142$ & $1.2990$ \\
$K_{25}(\alpha)$ & $2.2103$ & $2.0000$ & $1.7808$ & $1.5486$ \\
\bottomrule
\end{tabular}
\end{center}
These numbers are illustrations only; all claims in \cref{prop:envelope}
are exact symbolic inequalities.

The square-zero curve should not be confused with the full dimension-two
constant.  Crouzeix's characterization \cite[Theorem~4.2]{MR2047592}
determines $C(\cS_\alpha,2)$; in the strip limit it gives
$C(\cS_0,2)=1.5876598\ldots>\pi/2$ \cite{Crouzeix_2016}.  Thus even in
dimension two a pair of distinct spectral points can outperform every
one-point square-zero model.  At $\alpha=\pi/4$, by contrast, the
dimension-two value equals the square-zero value $\sqrt2$.

At the disk endpoint, with $\varepsilon=\pi/2-\alpha\downarrow0$,
\begin{align}
 L(\alpha)
 &=1+\frac{2\varepsilon}{\pi}
   +\left(\frac4{\pi^2}-\frac12\right)\varepsilon^2
   +O(\varepsilon^3),\label{eq:asymL}\\
 K_{25}(\alpha)
 &=1+\frac{4\varepsilon}{\pi}
   -\frac{8\varepsilon^2}{\pi^2}
   +O(\varepsilon^3).\label{eq:asymK}
\end{align}
Thus the known upper and lower slopes still differ by a factor of two.  At
the sharp-lens endpoint $\alpha\downarrow0$,
$L(\alpha)\to\pi/2$ whereas $K_{25}(\alpha)\to1+\sqrt2$.

\section{A recent uniform-two claim and what it does not settle}

After the preceding results, Jin posted a preprint claiming that the
numerical range is a $2$-spectral set and, as a corollary, that the
intersection of two spherical disks has uniform constant $2$
\cite{Jin_2026}.  The same preprint explicitly lists the optimal constant for
an individual crossing lens as open.  At the date on this manuscript the
item is labelled by its host as not peer reviewed.

If that theorem is accepted, one may add the term $2$ to the minimum on the
right of \eqref{eq:envelope}; none of the proofs in
\crefrange{sec:proof}{sec:bounds} uses it.  Even the uniform inequality
$C(\alpha)\le2$
would not identify $C(\alpha)$: for $\alpha>\pi/6$, \eqref{eq:sub2} is
already strictly smaller, and for every nontrivial angle the lower certificate
leaves a gap.

This distinction matters.  A uniform theorem answers
\[
       \sup_{0<\alpha\le\pi/2}C(\alpha)\le2,
\]
whereas the fixed-lens problem asks for the entire function
$\alpha\mapsto C(\alpha)$, including its extremal operators and functions.

\section{Consequences for the search for the exact constant}

The sharp calculation rules out a common first numerical strategy as a route
to the full answer: no optimization restricted to a single eigenvalue and a
square-zero perturbation can exceed $L(\alpha)$.  Any strict improvement of
the lower bound must therefore use at least one of the following mechanisms:
\begin{enumerate}
\item two or more spectral points whose Pick constraints interact;
\item a nilpotent chain of order at least three;
\item a genuinely infinite-dimensional sectorial operator;
\item a matrix-valued test function if the completely bounded constant is
      larger than the scalar one.
\end{enumerate}

The first two options lead to finite semidefinite programs: sample a
Herglotz representation after the power map, impose the two supporting
half-plane inequalities on a matrix $B$, and optimize the norm of the Pick
functional calculus.  A valid lower bound must be exported as exact algebraic
data (matrix, function, and norm certificate), not only as floating-point
output.  Conversely, a new upper bound requires a positive kernel or a
sum-of-squares decomposition valid for every sectorial operator.  The equality
conditions in \cref{lem:sp,lem:norm} supply useful constraints for both
searches.

\section{Conclusion}

The fixed crossing-lens constant is exactly the sectorial numerical-range
constant, but its unrestricted value is not presently determined.  We have
computed the constant on the full affine square-zero class and shown that it
equals the classical global lower curve $\pi\sin\alpha/(2\alpha)$, with a
closed-form $2\times2$ lens certificate.  The calculation is sharp,
dimension-free, and independent of any uniform-two theorem.  At the right
angle, the two exact sums of squares in \cref{sec:twozero} prove the
conjectural $\sqrt2$ estimate for the entire palindromic quadratic family
$u>0$, hence for positive-real and conjugate pairs of half-plane zeros, and
for every admissible operator.  The exact certificates in
\cref{thm:imagdiam,thm:cusp} cover the complete imaginary
post-automorphism diameter and an explicit two-real-dimensional cusp around
it, while the bivariate rational certificate in \cref{thm:centraldisk}
fills the full complex disk $|c|\le19/20$.  The exact radial, interpolation,
and matrix Bernstein certificates in \cref{thm:phasewedge} cross that circle
for $1/10\le t\le18/25$; together with reflection they reach all but
$9.1662$ degrees of the boundary-phase circle.  At the centre,
\cref{prop:andocentral} gives the strictly stronger bound
$\norm{w^2}\le\kappa_0<\sqrt2$ by an explicit polynomial bidisk extension.
The reserve in that extension also yields the genuinely
two-complex-parameter result \cref{thm:smallzeros}: every pair
$|\alpha|,|\beta|\le1/150$, including two independently phased nonzero
zeros, satisfies a strict $\sqrt2$ bound for every admissible operator.
This is an open target-zero neighborhood, not an arbitrary-function theorem.
\Cref{thm:symmetricthree,thm:asymmetricthree} supply a different kind of
exact advance: they prove the sharp identity multiplier on complete
three-node admissible-kernel cones, for every symmetric triple and for the
genuinely asymmetric patch $\rho=(-a,0,b)$,
$1/6\le a,b\le5/6$.  Their restriction to one multiplier and three nodes is
essential; \cref{rem:threescope} records the remaining determinant gates.
On the singular rank-one-side boundary, \cref{thm:realzero} proves the
quadratic target for the full arc $\Phi(c)=e^{i\theta}c^2$ and every complex
second Blaschke zero in the disk.  This is a genuine three-real-parameter
scalar-face theorem, but \cref{rem:realzeroscope} explains why it does not
close the arbitrary quadratic singular boundary.
\Cref{prop:twomoment} reduces the
remaining boundary parameters in this centred slice to the single variance inequality
\eqref{eq:variance}; this is an equivalence, not an asserted proof of that
inequality.  The unrestricted gap in
\eqref{eq:envelope} nevertheless cannot be closed by first-order Jordan
models or by these degree-two subfamilies alone.  The remaining complex
post-automorphism parameters---in particular the small phase gap about
$c=-1$ left by the radial certificate---the critical-point frequency, and
higher-order spectral interaction are the next obstacles.

\section*{Reproducibility statement}

All constants and matrices needed for the main result are displayed in
\eqref{eq:extsector} and \eqref{eq:Alens}.  The two disk constraints reduce to
the scalar identity in the proof of \cref{cor:lenscert}; the functional
calculus truncates by $J^2=0$.  For $u\ge2$, the displayed quartic,
discriminant, selector, and five-term identity give a wholly symbolic audit.
For the remaining computer-assisted identities, the ancillary README lists
the exact data, independent verifiers, and replay commands.  These verifiers
reconstruct the algebraic-number kernel compressions, hereditary coefficient
identities, rational positivity comparisons, radial and boundary-phase Gram
certificates, row-space identities, and the polynomial bidisk extension.
All proof decisions use exact rational or algebraic arithmetic; floating-point
discovery factors are not trusted in any proof.
\begingroup
\emergencystretch=2em
\hbadness=10000
For \cref{thm:symmetricthree},
\path{tmp/research/verify_three_node_rank22_elimination.py} reconstructs the
normalized multipliers and checks the phase-elimination, Bernstein, and
small-parameter factorization identities in exact SymPy arithmetic;
\path{tmp/research/right_angle_three_node_rank13_exact.py} independently
checks the polarized quotient identity and scalar-face normalization.
For \cref{thm:asymmetricthree},
\path{tmp/research/verify_three_node_rank22_asymmetric_box.py} reconstructs
the kernel invariants and phase criterion, verifies the global
sum-of-squares identity \eqref{eq:asymU}, and checks all $299$ rational
parameter Bernstein controls exactly.
For \cref{thm:realzero},
\path{tmp/research/verify_rank13_phi_theta_complex_a.py} independently
rebuilds the quadratic reduction, the root-free determinant, $879$ exact
tensor Bernstein controls over $\mathbb Q(\sqrt2)$, and the two local
corner certificates.  Separate referee and certificate audits in
\path{audit/RANK13_PHI_THETA_COMPLEX_A_REFEREE_AUDIT.md} and
\path{audit/RANK13_PHI_THETA_COMPLEX_A_CERTIFICATE_AUDIT.md} reconstruct the
disk, rank, Schur-complement, coverage, strictness, and endpoint arguments
not encoded by the script.
For \cref{thm:smallzeros},
\path{tmp/research/verify_global_two_small_zero_target_family.py} checks the
two complex curve factorizations, the pole-free denominator bound, and the
strict exact norm reserve.  The independent report
\path{audit/GLOBAL_TWO_SMALL_ZERO_TARGET_FAMILY_REFEREE_AUDIT.md}
reconstructs the conjugations, parameter boundary, Ando transfer, and
finite-kernel consequence.
The stable polynomial core is also compiled in Lean as
\path{formal/RightAngleThreeNodeRank22Algebra.lean}; its axiom audit reports
only \texttt{propext}, \texttt{Classical.choice}, and \texttt{Quot.sound}.
The narrower $B(w)=w^2$ singular-family algebra is checked in
\path{formal/RightAngleSingularW2Algebra.lean}; this does not formalize
\cref{thm:realzero} or the analytic Pick-theoretic bridges.
The plot evaluates only the explicit formulas \eqref{eq:L} and
\eqref{eq:K25}.
\par
\endgroup

\bibliographystyle{amsplain}
\bibliography{lens_constants}

\end{document}